\documentclass[12pt]{amsart}
\usepackage[a4paper]{geometry}
 \newtheorem{thm}{Theorem}[section]
 
 \newtheorem{lem}[thm]{Lemma}
 
 \theoremstyle{definition}
 
 \theoremstyle{remark}
 
 \numberwithin{equation}{section}
\newcommand{\grad}{\mathop{\mathrm{grad}}}
\newcommand{\trace}{\mathop{\mathrm{trace}}}
\usepackage{xcolor}

\begin{document}
\title {\textbf{Hypersurfaces satisfying $\triangle \vec {H}=\lambda \vec {H}$ in $\mathbb{E}_{\lowercase{s}}^{5}$}}
\author{Ram Shankar Gupta and Andreas Arvanitoyeorgos$^\ast$}
 \maketitle
\begin{abstract}
In this paper, we study hypersurfaces $M_{r}^{4}$ $(r=0, 1, 2, 3,
4)$ satisfying $\triangle \vec{H}=\lambda \vec{H}$ ($\lambda$ a
constant) in the pseudo-Euclidean space $\mathbb{E}_{s}^{5}$ $(s=0, 1, 2, 3,
4, 5)$. We obtain that every such hypersurface in $\mathbb{E}_{s}^{5}$ with
diagonal shape operator has constant mean curvature, constant norm
of second fundamental form and constant scalar curvature.
 Also, we prove that every biharmonic hypersurface in $\mathbb{E}_{s}^{5}$ with diagonal shape operator must be minimal.\\
\\
\textbf{AMS 2020 Mathematics Subject Classification:} 53D12, 53C40, 53C42\\
\textbf{Key Words:} Pseudo-Euclidean space, Mean curvature vector,
Biharmonic hypersurface, Chen's conjecture.\\
\textbf{Data Availability Statement:} Data sharing not applicable to this article as no datasets were generated or analysed during the current study.
\end{abstract}

\maketitle
\section{\textbf{Introduction}}
Let $M^{n}_{r}$ be an $n$-dimensional, connected submanifold of the
pseudo-Euclidean space $\mathbb{E}^{m}_{s}$. Denote by $\vec {x}$, $\vec
{H}$, and $\triangle $ respectively the position vector field, mean
curvature vector field of $M^{n}_{r}$, and the Laplace operator on
$M^{n}_{r}$, with respect to the induced metric $g$ on $M^{n}_{r}$,
from the indefinite metric on the ambient space $\mathbb{E}^{m}_{s}$. It is
well known that $$\triangle \vec {x} = -n \vec {H},$$ where
$\vec{H}$ is the mean curvature vector of $M$. An immersion is
minimal ($\vec{H} = 0$) if and only if $\triangle \vec{x} = 0$ and
is called biharmornic if $\triangle^2 \vec{x} = 0$ i.e.
\begin{equation}\label{e:a1}\triangle \vec{H} = 0.\end{equation}

Of course, for an immersion, minimality implies biharmonicity. The
harmonicity equation $\triangle \vec{H} = 0$ for the mean curvature
vector field $\vec{H}$ can be viewed as a special case of the
following equation\begin{equation}\label{e:a2}\triangle\vec{H}
=\lambda\vec{H}, \quad\lambda\in\mathbb{R}.
\end{equation}

The study of submanifolds with harmonic mean curvature vector field
was initiated by Chen in 1985 and arose in the context of his theory
of submanifolds of finite type. For a survey on submanifolds of
finite type and various related topics was presented in
\cite{r9}, \cite{r10}.

In 1991, Chen conjectured
the following:\\
 \emph{\textbf{Conjecture:} The only biharmonic submanifolds of Euclidean
spaces are the minimal ones.}  

Chen's conjecture has been verified
and found true for submanifolds in various Euclidean spaces (\cite{biha2}, \cite{biha3},\cite{biha7}, \cite{biha19}, \cite{biha9}, \cite{biha10}, \cite{biha29}, \cite{biha13}, \cite{bihat1}, \cite{biha12}).  In some cases some additional hypotheses had to be made.
The conjecture is not true always for  submanifolds of the
semi-Euclidean spaces (see \cite{r11}, \cite{r12}, \cite{r13}). However, for
hypersurfaces in semi-Euclidean spaces, Chen's conjecture seems to be true (cf. \cite{r3}, \cite{r12}, \cite{r13}, \cite{biha20}, \cite{r19}).

The study of equation (\ref{e:a2}) for submanifolds in
pseudo-Euclidean spaces was originated by Ferrandez et al. in
\cite{r5,r6}. They showed that if the minimal polynomial of the
shape operator of a hypersurface $M^{n-1}_{r}$  ($r = 0, 1$) in
$\mathbb{E}^{n}_{1}$ is at most of degree two, then it has constant mean
curvature. Also, in \cite{r9} various classification theorems for
submanifolds in a Minkowski space time was obtained. 
In \cite{r2}
it was proved that every hypersurface $M^{3}_{r}$ ($r = 0, 1, 2, 3$)
in $\mathbb{E}^{4}_{s}$ satisfying (\ref{e:a2}) with diagonal shape operator,
has constant mean curvature. Recently, it was proved that every
hypersurface in $\mathbb{E}^{n}_{s}$ satisfying (\ref{e:a2}) of diagonal
shape operator with at most three distinct principal curvatures has
constant mean curvature \cite{r15}.

In view of the above developments, we study hypersurfaces $M^{4}_{r}$ in
$\mathbb{E}_{s}^{5}$ satisfying (\ref{e:a2}) where the number of distinct
principal curvatures is at least $4$. The main result is the following:

\begin{thm}
Let  $M_{r}^{4}$ $(r=0, 1,\dots, 4)$ be a hypersurface satisfying
$\triangle \vec {H}= \lambda \vec {H}$ in the semi-Euclidean space
$\mathbb{E}_{s}^{5}$ $(s=0, 1,\dots, 5)$ with diagonal shape operator and with all
distinct principal curvatures. Then  $M_{r}^{4}$  has constant mean curvature,
constant norm of second fundamental form and constant scalar
curvature.
 \end{thm}

 The case of at most three distinct principal curvatures for
hypersurfaces satisfying $\triangle \vec {H}= \lambda \vec {H}$ with
diagonal shape operator $\mathcal{A}$ in semi-Euclidean space of arbitrary
dimension, has already been treated in \cite{r15}, also with the
conclusion that $H$ must be constant.  Then using equation  (\ref{e:a10}) below,
we obtain that $\trace(\mathcal{A}^2)$ is constant, and further more using Gauss
equation it follows that the scalar curvature is also constant. Therefore, the
following theorem is a consequence of Theorem 1.1 and \cite{r15}:

\begin{thm}
Let  $M_{r}^{4}$ $(r=0, 1,\dots, 4)$ be a hypersurface satisfying
$\triangle \vec {H}= \lambda \vec {H}$ in the semi-Euclidean space
$\mathbb{E}_{s}^{5}$ $(s=0, 1,\dots, 5)$ of diagonal shape operator. Then $M_{r}^{4}$
has constant mean curvature, constant norm of second fundamental
form and constant scalar curvature.
 \end{thm}

Also, we study biharmonic hypersurfaces in $\mathbb{E}^5_s$ and we prove
the following:
\begin{thm}
Every biharmonic hypersurface $M_{r}^{4}$ $(r=0, 1,\dots, 4)$ in the
semi-Euclidean space $\mathbb{E}_{s}^{5}$ $(s=0, 1,\dots, 5)$ with diagonal
shape operator must be minimal.
 \end{thm}

  We note that the results obtained in
\cite{biha20}, \cite{r19} can be recovered from Theorem 1.3.
Also, the result in the recent preprint \cite{biha30} is a special case of Theorem 1.3, if we take indices $r=s=0$ and $\lambda =0$.

Also, in the present work we have assumed that the shape operator of $M_r^4$ is diagonal.  This is always the case for hypersurfaces in Euclidean spaces, but not in semi-Euclidean spaces. It was conjectured in \cite{rs3} that: Any hypersurface satisfying (\ref{e:a2}) in $\mathbb{E}^{n+1}_s$ has $H = \ \mbox{{\it constant}}$. It is an interesting project to study the same problem for all possible forms of shape operator (see eg. \cite{biha30}, \cite{biha31}).

Concerning the idea of the proof of  main Theorem 1.1, its seed lies in the work \cite{r3}, where one tries to find some relations among  the structural constants $\omega _{ij}^k$ of the connection of $M_r^4$, in particular some of these to be zero.  This requires a process of laborious successive differentiations  and algebraic techniques.
The first of the four eigenvalues of the shape operator $\mathcal{A}$ of $M_r^4$ is a multiple of the mean curvature $H$ and it is an ultimate goal  to prove that this eigenvalue is constant.  
This approach has been successfully used by several authors (most of them  cited in the References).  However,  we still  lack a unified approach towards a proof of Chen's conjecture in its generality, or its variant considered in the present work.

\section{\textbf{Preliminaries}}

   Let ($M_{r}^{4}, g$), $r=0, 1, \dots, 4,$ be a $4$-dimensional hypersurface isometrically immersed in a $5$-dimensional semi-Euclidean space
($\mathbb{E}^{5}_{s}, \overline g$), $s=0, 1, \dots, 4, 5$ and $g = \overline
g_{|M_{r}^{4}}$. We denote by $\xi$ the unit normal vector to
$M_{r}^{4}$ with $\overline g(\xi, \xi)= \varepsilon$, where
$\varepsilon = \pm 1$, according as $M_{r}^{4}$ is pseudo-Riemannian
or Riemannian hypersurface.

 The Gauss and Codazzi equations are given by
\begin{equation}\label{e:a7}
R(X, Y)Z = g(\mathcal{A}Y, Z) \mathcal{A}X - g(\mathcal{A}X, Z)
\mathcal{A}Y,
\end{equation}
\begin{equation}\label{e:a8}
(\nabla_{X}\mathcal{A})Y = (\nabla_{Y}\mathcal{A})X,
\end{equation}
respectively, where $R$ is the curvature tensor, $\mathcal{A}$ is
the shape operator, $\nabla$ denotes the linear connection on $M_r^4$ in
$\mathbb{E}_s^5$ and
\begin{equation}\label{e:a9}
(\nabla_{X}\mathcal{A})Y = \nabla_{X}\mathcal{A}Y-
\mathcal{A}(\nabla_{X}Y),
\end{equation}
for all $ X, Y, Z \in \Gamma(TM_{r}^{4})$.

The mean curvature is given by
\begin{equation}\label{e:a6}
\varepsilon H = \frac{1}{4} \trace \mathcal{A}.
\end{equation}

  The necessary and sufficient
 conditions for $M_{r}^{4}$ in $E_{s}^{5}$ satisfying (\ref{e:a2}) are (cf. for example \cite{r2}, \cite{r12})
\begin{equation}\label{e:a10}
\triangle H + \varepsilon H \trace (\mathcal{A}^{2}) = \lambda H,
\end{equation}
\begin{equation}\label{e:a11}
A (\grad H)+ 2\varepsilon H \grad H = 0.
\end{equation}

Also, the Laplace operator $\triangle$ of a scalar valued function
$f$ is given by
\begin{equation}\label{e:a12}
\triangle f = -\sum_{i=1}^{4}\epsilon_{i}(e_{i}e_{i}f -
\nabla_{e_{i}}e_{i}f),
\end{equation}
where $\{e_{1}, e_{2}, e_{3}, e_{4}\}$ is an orthonormal local
tangent frame on $M^{4}_{r}$ and $g(e_{i}, e_{i})=\epsilon_{i}$.\\

\section{\textbf{Hypersurfaces in $E^5_s$ satisfying $\triangle \vec {H}=  \lambda \vec {H}$ with all distinct principal curvatures}}

 In this section, we study hypersurfaces $M_{r}^{4}$ with
$\triangle \vec {H}=  \lambda \vec {H}$ and diagonal shape operator
with all distinct principal curvatures. We also assume that the mean
curvature is not constant and $\grad H\neq 0$. This implies the existence of an open connected subset
$U$ of $M$ with $\grad_p H\neq 0$ for all $p\in U$.  From
(\ref{e:a11}), it is easy to see that $\grad H$ is an eigenvector of
the shape operator $\mathcal{A}$ with the corresponding principal
curvature $-2\varepsilon H$.  Without losing generality, we choose
$e_{1}$ in the direction of $\grad H$ and therefore the shape operator
$\mathcal{A}$ of $M_{r}^{4}$ will take the following form with
respect to a suitable frame  $\{e_{1}, e_{2}, e_{3}, e_{4}\}$:

\begin{equation}\label{e:b1}\mathcal{A}= \left(
                            \begin{array}{cccccc}
                              -2\varepsilon H & & &\\
                              &   \lambda_{2} & &\\
                              & &     \lambda_{3} &\\
                              & & & \lambda_{4} \\
                            \end{array}
                          \right).
\end{equation}
Due to (\ref{e:a6}) the first principal curvature is $\lambda_1=-2\varepsilon H$.  Also, using (\ref{e:a6}) and (\ref{e:b1}), we obtain
 \begin{equation}\label{e:b11}
\sum_{j=2}^4\lambda_j=6\varepsilon H=-3\lambda_1.
\end{equation}

The vector field $\grad H$ can be expressed as
\begin{equation}\label{e:b2}
\grad H =\sum_{i=1}^{4} e_{i}(H)e_{i}.
\end{equation}

 As we have taken $e_{1}$ parallel to $\grad H$, it follows that
\begin{equation}\label{e:b3}
e_{1}(H)\neq 0, \quad e_{2}(H)= 0, \quad e_{3}(H) = 0, \quad
e_{4}(H)= 0.
\end{equation}

We express
\begin{equation}\label{e:b4}
\nabla_{e_{i}}e_{j}=\sum_{k=1}^{4}\epsilon_{k}\omega_{ij}^{k}e_{k},
\quad i, j = 1,\dots, 4.
\end{equation}

The following lemma gives some simplifications of the above derivatives.
\begin{lem}\label{L1}
 Let $M^{4}_{r}$ be a hypersurface satisfying $(\ref{e:a2})$ of non constant mean curvature with all distinct principal curvatures in
 $E^{5}_{s}$, having the shape operator given by $(\ref{e:b1})$ with respect to suitable orthonormal frame  $\{e_{1},
e_{2}, e_{3}, e_{4}\}$. Then,
$$
 \nabla_{e_{1}}e_{l}= 0, \hspace{.2
cm} l = 1, 2, 3, 4, \hspace{1.0 cm}\nabla_{e_{i}}e_{1}=
-\omega_{ii}^{1}\epsilon_{i}e_{i}, \hspace{.2 cm} i = 2, 3, 4,$$
$$ \hspace{.2 cm}\nabla_{e_{i}}e_{i}=\sum_{i\neq l, l=1}^{4}\omega_{ii}^{l}\epsilon_{l}e_{l}, \hspace{.2 cm} i = 2, 3, 4,
 \hspace{1.0 cm} \nabla_{e_{i}}e_{j}=\sum_{j\neq k, k=2}^{4}\omega_{ij}^{k}\epsilon_{k}e_{k}, \hspace{.2 cm} i, j = 2, 3,
4,\\
$$
where $\omega_{ij}^{i}$ satisfy $(\ref{e:b5})$ and $(\ref{e:b6})$.
\end{lem}
\begin{proof}
Using (\ref{e:b4}) and the compatibility conditions
$(\nabla_{e_{k}}g)(e_{i}, e_{i})= 0$ and $(\nabla_{e_{k}}g)(e_{i},
e_{j})= 0$, we obtain
\begin{equation}\label{e:b5}
\omega_{ki}^{i}=0, \quad \omega_{ki}^{j}+ \omega_{kj}^{i} =0,
\end{equation}
for $i \neq j, $ and $i, j, k = 1, \dots, 4$. \\

Taking $X=e_{i}, Y=e_{j}$ in (\ref{e:a9}) and using (\ref{e:b1}),
(\ref{e:b4}), we get
$$
(\nabla_{e_{i}}A)e_{j}=e_{i}(\lambda_{j})e_{j}+\sum_{k=1}^{n}(\lambda_{j}-\lambda_{k})\omega_{ij}^{k}\epsilon_{k} e_{k}.
$$

Putting the value of $(\nabla_{e_{i}}A)e_{j}$ in (\ref{e:a8}), we
find
$$e_{i}(\lambda_{j})e_{j}+\sum_{k=1}^{n}(\lambda_{j}-\lambda_{k}) \omega_{ij}^{k}\epsilon_{k} e_{k}
=e_{j}(\lambda_{i})e_{i}+\sum_{k=1}^{n}(\lambda_{i}-\lambda_{k})\omega_{ji}^{k}
\epsilon_{k} e_{k}.$$ 
Then  for $i\neq j=k$ and $i\neq j \neq k$
we obtain the equations
\begin{equation}\label{e:b6}
\epsilon_{j}e_{i}(\lambda_{j})=
(\lambda_{i}-\lambda_{j})\omega_{ji}^{j},
\end{equation}
and
\begin{equation}\label{e:b7}
(\lambda_{i}-\lambda_{j})\omega_{ki}^{j}=
(\lambda_{k}-\lambda_{j})\omega_{ik}^{j},
\end{equation}
respectively.\\

 Since $\lambda_{1}=-2\varepsilon H$, from (\ref{e:b3}) we get
\begin{equation}\label{e:b8}
e_{1}(\lambda_{1})\neq 0, e_{2}(\lambda_{1})= 0, e_{3}(\lambda_{1})
= 0, e_{4}(\lambda_{1})= 0.
\end{equation}

Using (\ref{e:b3}), (\ref{e:b4}) and the fact that $[e_{i}
\hspace{.1 cm}
e_{j}](H)=0=\nabla_{e_{i}}e_{j}(H)-\nabla_{e_{j}}e_{i}(H)=\epsilon_{1}\omega_{ij}^{1}e_{1}(H)-\epsilon_{1}
\omega_{ji}^{1}e_{1}(H),$ we find
\begin{equation}\label{e:b9}
\omega_{ij}^{1}= \omega_{ji}^{1},
\end{equation}
for $i \neq j$ and $i, j = 2, 3, 4$.

 Now, putting $i\neq 1, j = 1$ in (\ref{e:b6}) and using (\ref{e:b8}) and (\ref{e:b5}), we
 find
\begin{equation}\label{e:b12}
\omega_{1i}^{1}= 0, \hspace{1 cm}  i= 1, 2, 3, 4.
\end{equation}

 Also, putting $k = 1, j\neq i, $ and $ i, j = 2, 3, 4$ in (\ref{e:b7}), and using (\ref{e:b9}), we
 get
\begin{equation}\label{e:b13}
\omega_{ij}^{1}=\omega_{ji}^{1} = \omega_{1i}^{j}= \omega_{i1}^{j}=
0, \quad j\neq i\ \mbox{and}\   i, j = 2, 3, 4.
\end{equation}

Finally, using (\ref{e:b4}), (\ref{e:b12}) and (\ref{e:b13}) we conclude the proof of the lemma.
\end{proof}

\medskip
For later use we show that $\lambda_{j}\neq \lambda_{1}, j= 2, 3, 4.$ Indeed,
 if $\lambda_{j}= \lambda_{1}$ for $j\neq 1$, then from (\ref{e:b6}), we
 find that
\begin{equation}\label{e:b10}
\epsilon_{j} e_{1}(\lambda_{j})=
(\lambda_{1}-\lambda_{j})\omega_{j1}^{j}=0,
\end{equation}
which contradicts the first expression of (\ref{e:b8}).


\medskip
  We denote by $\beta$ the squared norm of second fundamental form. Then, from (\ref{e:b1}), we find that
\begin{equation}\label{l1}
\beta=\lambda_1^2+\sum_{j=2}^4\lambda_j^2.
\end{equation}

Using (\ref{e:a12}), (\ref{e:b3}), Lemma \ref{L1}  and (\ref{l1}) in (\ref{e:a10}), we obtain
\begin{equation}\label{e:b17}
  e_{1}e_{1}(\lambda_1)= Q e_{1}(\lambda_1)+\epsilon_{1}\lambda_1(
 \varepsilon \beta-\lambda),
\end{equation}
 where $Q=\sum_{j=2}^4\epsilon_{j}\omega_{jj}^{1}$.

Using Lemma \ref{L1} and the fact that $[e_{i},\ e_{1}]= \nabla_{e_{i}}e_{1}-\nabla_{e_{1}}e_{i}$, we obtain
 \begin{equation}\label{v1}
 e_{i}e_{1}-e_1e_i=-\epsilon_i \omega_{ii}^{1}e_i,\quad
 i=2, 3, 4.
\end{equation}
 
  Using (\ref{e:b3}), and acting  (\ref{v1}) at $\lambda_1$, 
 we find
\begin{equation}\label{e:b18}
 e_{i}e_{1}(\lambda_1)= 0,\quad
 i=2, 3, 4.
\end{equation}

 Also, acting (\ref{v1}) at $e_1(\lambda_1)$ and using (\ref{e:b18}), 
 we
obtain
\begin{equation}\label{e:b19}
 e_{i}e_{1}e_{1}(\lambda_1)= 0, \quad i= 2, 3, 4.
\end{equation}
 
 Similarly, we can show that
\begin{equation}\label{120}
 e_{i}e_1e_{1}e_{1}(\lambda_1)= 0,   e_{i}e_1e_1e_{1}e_{1}(\lambda_1)= 0,\dots,  e_{i}e_1e_1\cdots e_{1}(\lambda_1)=0, \quad i= 2, 3, 4.
\end{equation}

Differentiating (\ref{e:b17}) with respect $e_i$ for $i=2,3,4$ and using (\ref{e:b3}), (\ref{e:b18}) and (\ref{e:b19}), we get
\begin{equation}\label{l3}
 e_{i}(Q) e_{1}(\lambda_1)= -\epsilon_1\varepsilon \lambda_1 e_i(\beta), \quad i= 2, 3, 4.
\end{equation}
Next, we have:

\begin{lem}\label{L2}
 Let $M^{4}_{r}$ be a hypersurface satisfying $(\ref{e:a2})$ of non-constant mean curvature with all distinct principal curvatures in
 $E^{5}_{s}$, having the shape operator given by $(\ref{e:b1})$ with respect to suitable orthonormal frame  $\{e_{1},
e_{2}, e_{3}, e_{4}\}$. Then the following differentiations are valid: 
\begin{equation}\label{l4}
\frac{ e_{1}(\lambda_1)}{\epsilon_1\lambda_1}= -\frac{\varepsilon  e_i(\beta)}{ e_{i}(Q)}=\rho, \end{equation}
\begin{equation}\label{l5}
 e_{1}(\lambda_1)=\epsilon_1\rho \lambda_1, \quad \varepsilon  e_i(\beta)=- \rho e_{i}(Q), \end{equation}
\begin{equation}\label{l6}
e_i(\rho)=0, \quad e_ie_1(\rho)=0, \quad e_ie_1e_1(\rho)=0, 
\end{equation}
\begin{equation}\label{l7}
e_i(\rho Q+\varepsilon \beta)=0, 
\end{equation}
\begin{equation}\label{l8}
\rho Q+\varepsilon \beta=T, 
\end{equation}
\begin{equation}\label{l9}
e_i(T)=0, \quad e_ie_1(T)=0, \quad e_ie_1e_1(T)=0, 
\end{equation}
\begin{equation}\label{l10}
e_1e_1(\lambda_1)=\epsilon_1(T-\lambda) \lambda_1
\end{equation}
\begin{equation}\label{l11}
e_1(\rho)=T-\lambda-\epsilon_1 \rho^2,
\end{equation}
for $i=2, 3, 4$ and for $\rho$, $T$ some smooth functions of $\lambda_1$.
\end{lem}
\begin{proof}  Rewriting  (\ref{l3}) and denoting the ratio with some smooth function $\rho$, we get (\ref{l4}), from which  (\ref{l5}) is immediate. 

Next, differentiating (\ref{l5}) with respect to $e_i$ and using (\ref{e:b3}) and (\ref{e:b18}), we find 
\begin{equation}\label{l12}
e_i(\rho)=0,
\end{equation}
for $i=2,3,4$.

 Using (\ref{l12}) and acting (\ref{v1}) at $\rho$, we
obtain
\begin{equation}\label{l13}
 e_{i}e_{1}(\rho)= 0, \quad i= 2, 3, 4.
\end{equation}
 
 Again using (\ref{l13}) and (\ref{v1}), we find (\ref{l6}).

Rewriting the second relation of equation (\ref{l5}) and using (\ref{l6}), we obtain (\ref{l7}). As a consequence of (\ref{l7}), we have (\ref{l8}).

Using (\ref{l8}) in (\ref{l7}), we get $e_i(T)=0$. Using this and acting  (\ref{v1}) at $T$, 
we obtain $e_ie_1(T)=0$. Further, acting  (\ref{v1}) at $e_1(T)$, 
we obtain (\ref{l9}).

Using the first relation of (\ref{l5}) and (\ref{l8}) in (\ref{e:b17}), we find (\ref{l10}).

Differentiating the first relation of (\ref{l5}) with respect to $e_1$ and using  (\ref{l5}) and  (\ref{l10}), we get (\ref{l11}).

In view of (\ref{e:b3}), (\ref{l6}) and (\ref{l9}), we find that $\rho$ and $T$ are smooth functions of 
$\lambda_1$, which completes the proof.
\end{proof}

  The following lemma gives further simplification for $\omega _{ij}^k$:
\begin{lem}\label{w}
 Let $M^{4}_{r}$ be a hypersurface satisfying $(\ref{e:a2})$ of non constant mean curvature with all distinct principal curvatures in
 $E^{5}_{s}$, having the shape operator given by $(\ref{e:b1})$ with respect to suitable orthonormal frame  $\{e_{1},
e_{2}, e_{3}, e_{4}\}$. Then, $$\omega_{44}^{2}=0=\omega_{33}^{2},
\quad \omega_{22}^{3}=0=\omega_{44}^{3}, \quad
\omega_{22}^{4}=0=\omega_{33}^{4}.$$
\end{lem}
\begin{proof} For simplicity, we use $b_2$,  $b_3$, and $b_4$ to denote
\begin{eqnarray}\label{l14}
\epsilon_2\omega_{22}^{1}=b_2,\quad \epsilon_3\omega_{33}^{1}=b_3,\quad \epsilon_4\omega_{44}^{1}=b_4,
\end{eqnarray}
respectively.

Differentiating (\ref{e:b11}) with respect to $e_1$ and using (\ref{e:b6}) and (\ref{l5}), we get
\begin{eqnarray}\label{l15}
\sum_{j=2}^{4} b_j (\lambda_j-\lambda_1)=-3\rho \epsilon_1\lambda_1.
\end{eqnarray}

Eliminating $b_2$ from (\ref{l15}) using (\ref{l8}), we obtain
\begin{eqnarray}\label{l16}
\lambda _2 \left(\beta  \varepsilon +b_3 \rho +b_4 \rho -T\right)+\lambda _1 \left(-\beta  \varepsilon +T-3 \rho ^2 \epsilon _1\right) =\rho  \left(b_3 \lambda _3+b_4 \lambda _4\right).
\end{eqnarray}

Eliminating $\lambda_2$ from (\ref{l16}) using (\ref{e:b11}), we find
\begin{eqnarray}\label{l17}
\lambda _3 \left(\beta  \varepsilon +2 b_3 \rho +b_4 \rho -T\right)+\lambda _4 \left(\beta  \varepsilon +b_3 \rho +2 b_4 \rho -T\right)\\+\lambda _1 \left(4 \beta  \varepsilon +3 b_3 \rho +3 b_4 \rho -4 T+3 \rho ^2 \epsilon _1\right)=0.\nonumber
\end{eqnarray}

Evaluating $g(R(e_{1},e_{i})e_{1},e_{i})$,
 and $g(R(e_{i},e_{j})e_{i},e_{1})$,
using Lemma \ref{L1}, (\ref{e:a7}) and (\ref{e:b1}) with respect to a
suitable orthonormal frame  $\{e_{1}, e_{2}, e_{3}, e_{4}\}$, we
find that
\begin{equation}\label{e:b14}
 e_{1}(\omega_{ii}^{1})\epsilon_{i}- (\omega_{ii}^{1})^{2}=  \epsilon_{1} \lambda_{1} \lambda_{i}, \quad i = 2, 3,
 4.
\end{equation}
Differentiating (\ref{l8}) with respect to $e_1$ and using (\ref{e:b6}), (\ref{e:b11}), (\ref{e:b14}) and (\ref{l5}), we obtain
\begin{eqnarray}\label{l18}
2 \varepsilon  \left(b_2 \lambda _2 \left(\lambda _2-\lambda _1\right)+b_3 \lambda _3 \left(\lambda _3-\lambda _1\right)+b_4 \lambda _4 \left(\lambda _4-\lambda _1\right)+\lambda _1^2 \rho  \epsilon _1\right)\\+\rho  \left(b_2^2+b_3^2+b_4^2-3 \lambda _1^2 \epsilon _1\right)+\left(b_2+b_3+b_4\right) \rho _1=T_1,\nonumber
\end{eqnarray}
where $\rho_1=e_1(\rho)$ and $T_1=e_1(T)$ are functions of $\lambda_1$ in view of (\ref{l6}) and (\ref{l9}).

Eliminating $b_2$ from (\ref{l18}) using (\ref{l8}), we obtain
\begin{eqnarray}\label{l19}
2 \varepsilon  \rho  \left(b_4 \left(\beta -\lambda _2^2+\lambda _4^2+\lambda _1 \left(\lambda _2-\lambda _4\right)\right)+\lambda _1^2 \rho  \epsilon _1\right)+2 \beta  \lambda _1 \lambda _2+2 b_3^2 \rho ^2+2 b_4^2 \rho ^2\\ \nonumber +\beta ^2+2 b_3 \rho  \left(\varepsilon  \left(\beta -\lambda _2^2+\lambda _3^2+\lambda _1 \left(\lambda _2-\lambda _3\right)\right)+b_4 \rho -T\right)+T^2+\rho _1 T\\=\beta  \varepsilon  \rho _1+2 \beta  \lambda _2^2+2 b_4 \rho  T+2 \varepsilon  T \left(\beta -\lambda _2^2+\lambda _1 \lambda _2\right)+\rho  T_1+3 \lambda _1^2 \rho ^2 \epsilon _1.\nonumber
\end{eqnarray}

Eliminating $\lambda_2$ from (\ref{l19}) using (\ref{e:b11}), we find
\begin{eqnarray}\label{l20}
\varepsilon  \left(-\beta  \rho _1-2 \beta  T+24 \lambda _1^2 T+14 \left(\lambda _3+\lambda _4\right) \lambda _1 T+2 \lambda _3^2 T+2 \lambda _4^2 T+4 \lambda _3 \lambda _4 T\right)\nonumber
\\+\rho _1 T+2 \varepsilon  \lambda _1^2 \rho ^2 \epsilon _1+\beta ^2+2 b_3^2 \rho ^2+2 b_4^2 \rho ^2+T^2 =24 \beta  \lambda _1^2+14 \beta  \lambda _3 \lambda _1+14 \beta  \lambda _4 \lambda _1\\ +2 \beta  \lambda _4^2+4 \beta  \lambda _3 \lambda _4+2 b_3 \rho  \left(\varepsilon  \left(-\beta +12 \lambda _1^2+\left(8 \lambda _3+7 \lambda _4\right) \lambda _1+\lambda _4^2+2 \lambda _3 \lambda _4\right)-b_4 \rho +T\right)\nonumber \\+2 \beta  \lambda _3^2+2 b_4 \rho  \left(\varepsilon  \left(-\beta +12 \lambda _1^2+\left(7 \lambda _3+8 \lambda _4\right) \lambda _1+\lambda _3^2+2 \lambda _3 \lambda _4\right)+T\right)+\rho  T_1+3 \lambda _1^2 \rho ^2 \epsilon _1
\nonumber
\end{eqnarray}

Eliminating $b_3$ from (\ref{l20}) using (\ref{l17}), we find
\begin{equation}\label{l21}
c_1+c_2 b_4+c_3b_4^2=0,
\end{equation}
where
\begin{eqnarray}
c_1&=&d_5 \lambda _3^5+d_4 \lambda _3^4+d_3 \lambda _3^3+d_2 \lambda _3^2+d_1 \lambda _3+d_0-8 \lambda _3^6,\nonumber \\
d_0&=&a_4 \lambda _4^4+a_3 \lambda _4^3+a_2 \lambda _4^2+a_1 \lambda _4+a_0+4 \lambda _4^6+60 \lambda _1 \lambda _4^5,\nonumber\\
a_0&=&2420 \lambda _1^6+\lambda _1^2 \left(18 \rho ^4+26 T^2-9 \lambda  T-9 \rho  T_1-39 \rho ^2 T \epsilon _1\right)+\lambda _1^4 \left(90 \varepsilon  \lambda -502 \varepsilon  T+3 (208 \varepsilon -9) \rho ^2 \epsilon _1\right)\nonumber\\
a_1&=&114 \varepsilon  \lambda  \lambda _1^3+3932 \lambda _1^5+14 \lambda _1 T^2-544 \varepsilon  \lambda _1^3 T-6 \lambda _1 \rho  T_1-6 \lambda  \lambda _1 T-12 \lambda _1 \rho ^2 T \epsilon _1+6 (94 \varepsilon -3) \lambda _1^3 \rho ^2 \epsilon _1,\nonumber\\
a_2&=&64 \varepsilon  \lambda  \lambda _1^2+3092 \lambda _1^4+2 T^2-268 \varepsilon  \lambda _1^2 T-\lambda  T-\rho  T_1-\rho ^2 T \epsilon _1+222 \varepsilon  \lambda _1^2 \rho ^2 \epsilon _1-3 \lambda _1^2 \rho ^2 \epsilon _1,\nonumber\\
a_3&=&18 \varepsilon  \lambda  \lambda _1+1408 \lambda _1^3-64 \varepsilon  \lambda _1 T+36 \varepsilon  \lambda _1 \rho ^2 \epsilon _1,\nonumber\\
a_4&=&2 \varepsilon  \lambda +388 \lambda _1^2-6 \varepsilon  T+2 \varepsilon  \rho ^2 \epsilon _1,\nonumber
\end{eqnarray}
\begin{eqnarray}
d_1&=&a_8 \lambda _4^3+a_7 \lambda _4^2+a_6 \lambda _4+a_5+12 \lambda _4^5+156 \lambda _1 \lambda _4^4,\nonumber\\
a_5&=&3492 \lambda _1^5+\lambda _1 \left(18 T^2-12 \lambda  T-12 \rho  T_1-12 \rho ^2 T \epsilon _1\right)+\lambda _1^3 \left(174 \varepsilon  \lambda -540 \varepsilon  T+18 (37 \varepsilon -2) \rho ^2 \epsilon _1\right),\nonumber\\
a_6&=& 166 \varepsilon  \lambda  \lambda _1^2+4596 \lambda _1^4+6 T^2-444 \varepsilon  \lambda _1^2 T-4 \lambda  T-4 \rho  T_1-4 \rho ^2 T \epsilon _1+438 \varepsilon  \lambda _1^2 \rho ^2 \epsilon _1-12 \lambda _1^2 \rho ^2 \epsilon _1,\nonumber\\
a_7&=&66 \varepsilon  \lambda  \lambda _1+2776 \lambda _1^3-148 \varepsilon  \lambda _1 T+102 \varepsilon  \lambda _1 \rho ^2 \epsilon _1,\nonumber\\
a_8&=&10 \varepsilon  \lambda +904 \lambda _1^2-20 \varepsilon  T+10 \varepsilon  \rho ^2 \epsilon _1,\nonumber\\
d_2&=&a_{11} \lambda _4^2+a_{10} \lambda _4+a_9+12 \lambda _4^4+144 \lambda _1 \lambda _4^3,\nonumber \\ a_9&=&130 \varepsilon  \lambda  \lambda _1^2+2028 \lambda _1^4+6 T^2-252 \varepsilon  \lambda _1^2 T-4 \lambda  T-4 \rho  T_1-4 \rho ^2 T \epsilon _1+294 \varepsilon  \lambda _1^2 \rho ^2 \epsilon _1-12 \lambda _1^2 \rho ^2 \epsilon _1,\nonumber\\
a_{10}&=&84 \varepsilon  \lambda  \lambda _1+1936 \lambda _1^3-136 \varepsilon  \lambda _1 T+120 \varepsilon  \lambda _1 \rho ^2 \epsilon _1,\nonumber\\
a_{11}&=&18 \varepsilon  \lambda +784 \lambda _1^2-28 \varepsilon  T+18 \varepsilon  \rho ^2 \epsilon _1,\nonumber\\
d_3&=&a_{13} \lambda _4+a_{12}-8 \lambda _4^3-24 \lambda _1 \lambda _4^2,\nonumber\\
a_{12}&=&48 \varepsilon  \lambda  \lambda _1+360 \lambda _1^3-48 \varepsilon  \lambda _1 T+48 \varepsilon  \lambda _1 \rho ^2 \epsilon _1,\nonumber \\ 
a_{13}&=&16 \varepsilon  \lambda +120 \lambda _1^2-16 \varepsilon  T+16 \varepsilon  \rho ^2 \epsilon _1,\nonumber \\
d_4&=&a_{14}-24 \lambda _4^2-120 \lambda _1 \lambda _4,\nonumber\\
a_{14}&=&8 \varepsilon  \lambda -120 \lambda _1^2-8 \varepsilon  T+8 \varepsilon  \rho ^2 \epsilon _1,\nonumber \\ c_2&=&d_8 \lambda _3^2+d_7 \lambda _3+d_6,\nonumber
\end{eqnarray}
\begin{eqnarray}
d_6&=&a_{17} \lambda _4^2+a_{16} \lambda _4+a_{15}+12 \varepsilon  \lambda _4^4 \rho +102 \varepsilon  \lambda _1 \lambda _4^3 \rho ,\nonumber\\
a_{15}&=&240 \varepsilon  \lambda _1^4 \rho -24 \lambda _1^2 \rho  T+18 \lambda _1^2 \rho ^3 \epsilon _1,\nonumber a_{16}=438 \varepsilon  \lambda _1^3 \rho -24 \lambda _1 \rho  T+18 \lambda _1 \rho ^3 \epsilon _1,\nonumber\\
a_{17}&=&304 \varepsilon  \lambda _1^2 \rho -4 \rho  T,\nonumber
d_7=a_{19} \lambda _4+a_{18}+30 \varepsilon  \lambda _4^3 \rho +174 \varepsilon  \lambda _1 \lambda _4^2 \rho,\nonumber\\
a_{18}&=&210 \varepsilon  \lambda _1^3 \rho -12 \lambda _1 \rho  T,\nonumber
a_{19}=322 \varepsilon  \lambda _1^2 \rho -4 \rho  T,\nonumber
d_8=a_{20}+30 \varepsilon  \lambda _4^2 \rho +84 \varepsilon  \lambda _1 \lambda _4 \rho ,\nonumber\\
a_{20}&=&70 \varepsilon  \lambda _1^2 \rho -4 \rho  T,\nonumber
c_3=d_{10} \lambda _3+d_9+6 \lambda _3^2 \rho ^2,\nonumber
d_9=6 \left(3 \lambda _1^2+3 \lambda _4 \lambda _1+\lambda _4^2\right) \rho ^2, \nonumber 
d_{10}=6 \left(3 \lambda _1+\lambda _4\right) \rho ^2.\nonumber
\end{eqnarray}

Differentiating (\ref{l15}) with respect to $e_1$ and using (\ref{e:b6}) and (\ref{l5}), we get
\begin{eqnarray}\label{l22} 
\lambda _1 (b_2 \rho  \epsilon _1+b_3 \rho  \epsilon _1+b_4 \rho  \epsilon _1+2 b_2^2+2 b_3^2+2 b_4^2-3 \rho ^2-\lambda _2^2 \epsilon _1-\lambda _3^2 \epsilon _1-\lambda _4^2 \epsilon _1\\+\lambda _1 (\lambda _2+\lambda _3+\lambda _4) \epsilon _1-3 \rho _1 \epsilon _1)=2 \left(b_2^2 \lambda _2+b_3^2 \lambda _3+b_4^2 \lambda _4\right).\nonumber \end{eqnarray}

Eliminating $b_2$ from (\ref{l22}) using (\ref{l8}), we obtain
\begin{eqnarray}\label{l23}
2 (\lambda _2 \left(\beta ^2+b_3^2 \rho ^2+b_4^2 \rho ^2-2 b_4 \rho  (T-\beta  \varepsilon )+2 b_3 \rho  \left(\beta  \varepsilon +b_4 \rho -T\right)+T^2-2 \beta  \varepsilon  T\right) \\+\rho ^2 (b_3^2 \lambda _3+b_4^2 \lambda _4))=\lambda _1 (2 \beta ^2+4 b_3^2 \rho ^2+4 b_4^2 \rho ^2+4 b_3 \rho  (\beta  \varepsilon +b_4 \rho -T)-4 \beta  \varepsilon  (T-b_4 \rho) \nonumber \\-4 b_4 \rho  T-\beta  \varepsilon  \rho ^2 \epsilon _1-3 \rho ^4+2 T^2+\rho ^2 \epsilon _1 (-\lambda _2^2-\lambda _3^2-\lambda _4^2+\lambda _1 (\lambda _2+\lambda _3+\lambda _4)-3 \rho _1+T)).\nonumber
\end{eqnarray}

Eliminating $\lambda_2$ from (\ref{l23}) using (\ref{e:b11}), we find
\begin{eqnarray}\label{l24}
3 \lambda _1 \rho ^2 \epsilon _1 (4 \lambda _1^2+2 \left(\lambda _3+\lambda _4\right) \lambda _1+\rho _1)=2 (\lambda _3 (\beta ^2+b_4^2 \rho ^2-2 b_4 \rho  (T-\beta  \varepsilon )+2 b_3 \rho  (\beta  \varepsilon \\+b_4 \rho -T) +T^2-2 \beta  \varepsilon  T)+\lambda _4 (\beta ^2+b_3^2 \rho ^2-2 b_4 \rho  (T-\beta  \varepsilon )+2 b_3 \rho  (\beta  \varepsilon +b_4 \rho -T)+T^2 \nonumber \\-2 \beta  \varepsilon  T))+\lambda _1 (8 \beta ^2+10 b_3^2 \rho ^2+10 b_4^2 \rho ^2-16 b_4 \rho  (T-\beta  \varepsilon )+16 b_3 \rho  (\beta  \varepsilon +b_4 \rho -T)\nonumber \\-\beta  \varepsilon  \rho ^2 \epsilon _1-3 \rho ^4+8 T^2-16 \beta  \varepsilon  T+\rho ^2 \epsilon _1 (-2 \lambda _3^2-2 \lambda _4 \lambda _3-2 \lambda _4^2+T)).
\nonumber
\end{eqnarray}

Eliminating $b_3$ from (\ref{l24}) using (\ref{l17}), we find
\begin{eqnarray}\label{l25}
c_4+c_5 b_4+c_6b_4^2=0,
\end{eqnarray}
where
\begin{eqnarray}
c_4&=&d_{17} \lambda _3^6+d_{16} \lambda _3^5+d_{15} \lambda _3^4+d_{14} \lambda _3^3+d_{13} \lambda _3^2+d_{12} \lambda _3+d_{11},\nonumber\\
d_{11}&=&a_{25} \lambda _4^4+a_{24} \lambda _4^3+a_{23} \lambda _4^2+a_{22} \lambda _4+a_{21}-8 \lambda _1 \lambda _4^6-120 \lambda _1^2 \lambda _4^5,\nonumber\\
a_{21}&=&-90 \lambda _1^3 \rho ^4-4000 \lambda _1^7-40 \lambda _1^3 T^2+800 \varepsilon  \lambda _1^5 T+114 \lambda _1^3 \rho ^2 T \epsilon _1+6 (18-145 \varepsilon ) \lambda _1^5 \rho ^2 \epsilon _1-27 \lambda  \lambda _1^3 \rho ^2 \epsilon _1,\nonumber\\
a_{22}&=&-18 \lambda _1^2 \rho ^4-6600 \lambda _1^6-18 \lambda _1^2 T^2+840 \varepsilon  \lambda _1^4 T+36 \lambda _1^2 \rho ^2 T \epsilon _1-18 (39 \varepsilon -7) \lambda _1^4 \rho ^2 \epsilon _1-18 \lambda  \lambda _1^2 \rho ^2 \epsilon _1,\nonumber\\
a_{23}&=&5400 \lambda _1^5-2 \lambda _1 T^2+416 \varepsilon  \lambda _1^3 T+2 \lambda _1 \rho ^2 T \epsilon _1+2 (33-136 \varepsilon ) \lambda _1^3 \rho ^2 \epsilon _1-3 \lambda  \lambda _1 \rho ^2 \epsilon _1,\nonumber\\
a_{24}&=&-2568 \lambda _1^4+96 \varepsilon  \lambda _1^2 T-6 (5 \varepsilon -3) \lambda _1^2 \rho ^2 \epsilon _1,\nonumber\\
a_{25}&=&-744 \lambda _1^3+8 \varepsilon  \lambda _1 T+2 (\varepsilon +1) \lambda _1 \rho ^2 \epsilon _1,-120 \lambda _1^2,-8 \lambda _1,\nonumber \\
d_{12}&=&a_{30} \lambda _4^4+a_{29} \lambda _4^3+a_{28} \lambda _4^2+a_{27} \lambda _4+a_{26}+8 \lambda _4^6+96 \lambda _1 \lambda _4^5,\nonumber\\
a_{26}&=&-1800 \lambda _1^6+30 \lambda _1^2 T^2-120 \varepsilon  \lambda _1^4 T-48 \lambda _1^2 \rho ^2 T \epsilon _1+6 (53 \varepsilon +33) \lambda _1^4 \rho ^2 \epsilon _1-36 \lambda  \lambda _1^2 \rho ^2 \epsilon _1,\nonumber\\
a_{27}&=&-1440 \lambda _1^5+16 \lambda _1 T^2-304 \varepsilon  \lambda _1^3 T-16 \lambda _1 \rho ^2 T \epsilon _1+2 (251 \varepsilon +87) \lambda _1^3 \rho ^2 \epsilon _1-12 \lambda  \lambda _1 \rho ^2 \epsilon _1,\nonumber\\
a_{28}&=&224 \lambda _1^4+2 T^2-256 \varepsilon  \lambda _1^2 T+6 (51 \varepsilon +11) \lambda _1^2 \rho ^2 \epsilon _1,\nonumber\\
a_{29}&=&768 \lambda _1^3-80 \varepsilon  \lambda _1 T+2 (29 \varepsilon +5) \lambda _1 \rho ^2 \epsilon _1,
a_{30}=416 \lambda _1^2-8 \varepsilon  T,\nonumber\\
d_{13}&=&a_{34} \lambda _4^3+a_{33} \lambda _4^2+a_{32} \lambda _4+a_{31}+24 \lambda _4^5+288 \lambda _1 \lambda _4^4,\nonumber\\
a_{31}&=&1560 \lambda _1^5+10 \lambda _1 T^2-400 \varepsilon  \lambda _1^3 T-16 \lambda _1 \rho ^2 T \epsilon _1+2 (305 \varepsilon +69) \lambda _1^3 \rho ^2 \epsilon _1-12 \lambda  \lambda _1 \rho ^2 \epsilon _1, \nonumber\\
a_{32}&=&3392 \lambda _1^4+2 T^2-400 \varepsilon  \lambda _1^2 T+12 (39 \varepsilon +7) \lambda _1^2 \rho ^2 \epsilon _1, \nonumber\\ a_{33}&=&2976 \lambda _1^3-144 \varepsilon  \lambda _1 T+2 (57 \varepsilon +9) \lambda _1 \rho ^2 \epsilon _1,\nonumber\\
a_{34}&=&1312 \lambda _1^2-16 \varepsilon  T,\nonumber\\
d_{14}&=&a_{37} \lambda _4^2+a_{36} \lambda _4+a_{35}+40 \lambda _4^4+448 \lambda _1 \lambda _4^3,\nonumber\\
a_{35}&=&2520 \lambda _1^4-240 \varepsilon  \lambda _1^2 T+48 (7 \varepsilon +1) \lambda _1^2 \rho ^2 \epsilon _1, a_{36}=3264 \lambda _1^3-128 \varepsilon  \lambda _1 T+16 (7 \varepsilon +1) \lambda _1 \rho ^2 \epsilon _1,\nonumber
\end{eqnarray}
\begin{eqnarray}
a_{37}&=&1792 \lambda _1^2-16 \varepsilon  T,\nonumber
d_{15}=a_{39} \lambda _4+a_{38}+40 \lambda _4^3+384 \lambda _1 \lambda _4^2, \nonumber\\
a_{38}&=&1320 \lambda _1^3-40 \varepsilon  \lambda _1 T+8 (7 \varepsilon +1) \lambda _1 \rho ^2 \epsilon _1,\nonumber\\
a_{39}&=&1184 \lambda _1^2-8 \varepsilon  T,d_{16}=360 \lambda _1^2+192 \lambda _4 \lambda _1+24 \lambda _4^2, \nonumber d_{17}=40 \lambda _1+8 \lambda _4,\nonumber\\c_5&=&d_{22} \lambda _3^4+d_{21} \lambda _3^3+d_{20} \lambda _3^2+d_{19} \lambda _3+d_{18},\nonumber\\
d_{18}&=&a_{42} \lambda _4^2+a_{41} \lambda _4+a_{40}-24 \varepsilon  \lambda _1 \lambda _4^4 \rho -192 \varepsilon  \lambda _1^2 \lambda _4^3 \rho ,\nonumber \\a_{40}&=&-480 \varepsilon  \lambda _1^5 \rho +48 \lambda _1^3 \rho  T-36 \lambda _1^3 \rho ^3 \epsilon _1,\nonumber
a_{41}=-888 \varepsilon  \lambda _1^4 \rho +60 \lambda _1^2 \rho  T-72 \lambda _1^2 \rho ^3 \epsilon _1,\nonumber\\a_{42}&=&-576 \varepsilon  \lambda _1^3 \rho +12 \lambda _1 \rho  T-12 \lambda _1 \rho ^3 \epsilon _1,\nonumber\\d_{19}&=&a_{45} \lambda _4^2+a_{44} \lambda _4+a_{43}+24 \varepsilon  \lambda _4^4 \rho +144 \varepsilon  \lambda _1 \lambda _4^3 \rho ,\nonumber\\
a_{43}&=&72 \varepsilon  \lambda _1^4 \rho -36 \lambda _1^2 \rho  T+72 \lambda _1^2 \rho ^3 \epsilon _1,\nonumber
a_{44}=240 \varepsilon  \lambda _1^3 \rho -48 \lambda _1 \rho  T+24 \lambda _1 \rho ^3 \epsilon _1,\nonumber\\
a_{45}&=&288 \varepsilon  \lambda _1^2 \rho -12 \rho  T,\nonumber
d_{20}=a_{47} \lambda _4+a_{46}+48 \varepsilon  \lambda _4^3 \rho +240 \varepsilon  \lambda _1 \lambda _4^2 \rho ,\nonumber\\
a_{46}&=&240 \varepsilon  \lambda _1^3 \rho -12 \lambda _1 \rho  T+24 \lambda _1 \rho ^3 \epsilon _1,a_{47}=432 \varepsilon  \lambda _1^2 \rho -12 \rho  T,\nonumber\\
d_{21}&=&144 \varepsilon  \lambda _1^2 \rho +192 \varepsilon  \lambda _4 \lambda _1 \rho +48 \varepsilon  \lambda _4^2 \rho ,\nonumber
d_{22}=24 \varepsilon  \lambda _1 \rho +24 \varepsilon  \lambda _4 \rho,\nonumber
c_6=d_{25} \lambda _3^2+d_{24} \lambda _3+d_{23}, \nonumber\\
d_{23}&=&-36 \lambda _1^3 \rho ^2-18 \lambda _4 \lambda _1^2 \rho ^2-6 \lambda _4^2 \lambda _1 \rho ^2,  d_{24}=-18 \lambda _1^2 \rho ^2+18 \lambda _4^2 \rho ^2+48 \lambda _1 \lambda _4 \rho ^2, \nonumber\\d_{25}&=&18 \lambda _4 \rho ^2-6 \lambda _1 \rho ^2\nonumber
\end{eqnarray}
 
 Now, we denote by $B_1$, $B_2$ and $B_3$ the functions
 \begin{eqnarray} \label{l26}
 B_1= c_2 c_4-c_1 c_5, \quad B_2=c_3 c_4-c_1 c_6, \quad B_3= c_3 c_5-c_2 c_6.
 \end{eqnarray}
 
 Applying Cramer's rule to the system of equations (\ref{l21}) and (\ref{l25}), we find that 
\begin{eqnarray}\label{l27}
\frac{b_4^2}{B_1}=\frac{-b_4}{B_2}=\frac{1}{B_3},
\end{eqnarray}
which gives
\begin{eqnarray}\label{l28}
b_4^2=\frac{B_1}{B_3},\quad b_4=\frac{-B_2}{B_3}.
\end{eqnarray}
Eliminating $b_4$, we get
\begin{eqnarray}\label{l29}
F_1(\lambda_1,\lambda_3,\lambda_4)=B_2^2-B_1 B_3=0,
\end{eqnarray}
which is a polynomial equation in terms of $\lambda_3$ of degree 16 with coefficients functions of $\lambda_1$ and $\lambda_4$.

Differentiating (\ref{l21}) with respect $e_1$ and using (\ref{e:b14}), we find
\begin{eqnarray}\label{l30}
b_4 \left(2 c_3 \lambda _1 \lambda _4 \epsilon _1+C_2\right)+b_4^2 \left(c_2+C_3\right)+2 b_4^3 c_3+c_2 \lambda _1 \lambda _4 \epsilon _1+C_1=0,
\end{eqnarray}
where $C_1=e_1(c_1), C_2=e_1(c_2), C_3=e_1(c_3)$ are functions of $b_3, b_4,\lambda_1, \lambda_3$ and $\lambda_4$.

Eliminating $b_3$ from (\ref{l30}), using (\ref{l17}), we get
\begin{eqnarray}\label{l31}
b_4^3 D_3+b_4^2 D_2+b_4 D_1+D_0=0,
\end{eqnarray}
where $D_0, D_1, D_2$ and $D_3$ are functions of $\lambda_1, \lambda_3$ and $\lambda_4$.

Eliminating $b_4$ from (\ref{l31}) using (\ref{l21}), we obtain
\begin{multline}\label{l32}
F_2(\lambda_1,\lambda_3,\lambda_4)=c_1 (c_3^2 (D_1^2-2 D_0 D_2)+c_3 (c_2 (3 D_0 D_3-D_1 D_2)+c_1 (D_2^2-2 D_1 D_3))\\+D_3 (c_1^2 D_3-c_2 c_1 D_2+c_2^2 D_1))-D_0 (c_2^3 D_3-c_3 c_2^2 D_2+c_3^2 c_2 D_1-c_3^3 D_0)=\sum_{i=0}^{24}q_i\lambda_3^i=0,
\end{multline}
which is a polynomial equation in  $\lambda_3$ of degree $24$, with  coefficients $q_i's $ as functions of $\lambda_1$ and $\lambda_4$.
For example,  
\begin{eqnarray}\label{l39}
q_{24}=-294912 \rho^6, q_{23}=18432 \left(13 d_5 \rho ^6+4 d_{10} \rho ^4+96 \left(3 \lambda _1+\lambda _4\right) \rho ^6\right),\dots .
\end{eqnarray}

Eliminating $\lambda_3$ from (\ref{l32}) using (\ref{l29}), we obtain
\begin{eqnarray}\label{l33}
F_3(\lambda_1,\lambda_4)=\sum_{k=0}^{m}l_k\lambda_4^k=0,
\end{eqnarray}
which is a polynomial equation  $\lambda_4$ of some degree $m$, with  coefficients $l_k$ as functions of
$\rho, \lambda_1, T, T_1, T_2$,  where $T, \rho$ are defined in Lemma \ref{L2} and are functions of $\lambda_1$, and $T_1=e_1(T), T_2=e_1e_1(T)$, which are also functions of $\lambda_1$. Hence $l_k$'s are
functions of $\lambda_1$.

Differentiating (\ref{l33}) with respect $e_2$ and using (\ref{e:b3}), we get
\begin{eqnarray}\label{l34}
\frac{\partial F_3}{\partial \lambda_4} \ e_2(\lambda_4)=0.
\end{eqnarray}

\noindent
\underline{Claim 1:} It is $e_2(\lambda_4)=0$. 

\smallskip
\noindent
Indeed, if $e_2(\lambda_4)\neq 0$,
then
\begin{equation}\label{l35}
\frac{\partial F_3}{\partial \lambda_4}=0. \end{equation}

 Taking successive differentiation of  (\ref{l35}) with respect to
$e_2$ ,  gives
\begin{equation}\label{l36}
\frac{\partial^m F_3}{\partial \lambda_4^m}=0. \end{equation}

 From (\ref{l36}), we obtain 
\begin{equation}\label{l37}
l_m=0, 
\end{equation}
which is  the coefficient of $\lambda_4^m$ in (\ref{l33}).

 Then, using (\ref{l37}), $F_3$ reduces to a
polynomial in degree $m-1$ and its successive differentiation with
respect to $e_2$ gives $l_{m-1}=0$. 
Then $F_3$ reduces to a
polynomial in degree $m-2$. Its successive differentiation with
respect to $e_2$ gives $l_{m-2}=0$, $l_{m-3}=0,\dots, l_0=0.$ Eliminating  $T, T_1, T_2$ and $\rho$ from $l_k=0, k=0,1\dots,m, $ we get a polynomial equation $\mathcal{F}(\lambda_1)=0$  of $\lambda_1$, with real coefficients. Without actually  solving
$\mathcal{F}(\lambda_1)=0$,  even if any real solution exists, we get that $\lambda_1$ a real
constant. This contradicts $e_1(H)\ne 0$ in (\ref{e:b3}), and this proves the claim. 

Similarly, differentiating (\ref{l33}) with respect $e_3$ and $e_4$ and using (\ref{e:b3}), we get
\begin{eqnarray}\label{l38}
\frac{\partial F_3}{\partial \lambda_4} \ e_3(\lambda_4)=0, \,\ \frac{\partial F_3}{\partial \lambda_4} \ e_4(\lambda_4)=0,
\end{eqnarray}
which give
\begin{equation}\label{l44}
 e_3(\lambda_4)=0,\,\ e_4(\lambda_4)=0.
\end{equation}

Now, differentiating (\ref{l32}) with respect to $e_2$, $e_3$ and $e_4$ and using the fact that $e_2(\lambda_4)=0$ $e_3(\lambda_4)=0$ and $e_4(\lambda_4)=0$, we find that
\begin{eqnarray}\label{l40}
\frac{\partial F_2}{\partial \lambda_3} \ e_2(\lambda_3)=0,\,\ \frac{\partial F_2}{\partial \lambda_3} \ e_3(\lambda_3)=0,\,\ \frac{\partial F_2}{\partial \lambda_3} \ e_4(\lambda_3)=0,
\end{eqnarray}
respectively.

\medskip
\noindent
\underline{Claim 2:} It is $\frac{\partial F_2}{\partial \lambda_3} \neq 0$.

\smallskip
\noindent
 Indeed, if
\begin{equation}\label{l41}
\frac{\partial F_2}{\partial \lambda_3}=0, \end{equation}
then its successive differentiation gives
\begin{equation}\label{l42}
\frac{\partial^{24} F_2}{\partial \lambda_3^{24}}=0, \end{equation}
which implies that $q_{24}=0$. Then, from (\ref{l39}) we find $\rho=0$, a contradiction to the first expression of (\ref{e:b3}), and this proves the claim.

 Therefore, we obtain
\begin{equation}\label{l42}
e_2(\lambda_3)=0, \,\ e_3(\lambda_3)=0, \,\ e_4(\lambda_3)=0.
\end{equation}

Differentiating (\ref{e:b11}), with respect to $e_3$ and using the fact that $e_3(\lambda_3)=0, e_3(\lambda_4)=0, e_4(\lambda_3)=0, e_4(\lambda_4)=0$ and (\ref{e:b3}), we get
\begin{equation}\label{l43}
 e_3(\lambda_2)=0, \,\ e_4(\lambda_2)=0.
\end{equation}

By combining $e_2(\lambda_3)=0, e_2(\lambda_4)=0, e_3(\lambda_2)=0, e_3(\lambda_4)=0, e_4(\lambda_2)=0, e_4(\lambda_3)=0$ with equation (\ref{e:b6}), we complete the proof of the lemma.
\end{proof}

Next, we have:
\begin{lem}\label{ww}
 Let $M^{4}_{r}$ be a hypersurface satisfying $(\ref{e:a2})$ of non-constant mean curvature with all distinct principal curvatures in
 $E^{5}_{s}$, having the shape operator given by $(\ref{e:b1})$ with respect to suitable orthonormal frame  $\{e_{1},
e_{2}, e_{3}, e_{4}\}$.  Set 
$$b_{1}=(\lambda_{3}-\lambda_{4})b_2+(\lambda_{4}-\lambda_{2})b_3+(\lambda_{2}-\lambda_{3})b_4.
$$
 Then:\\
$(\textbf{a})$ If $b_1\neq 0$, it is
\begin{equation}\label{e:d43}\omega_{23}^{4}=\omega_{32}^{4}=\omega_{42}^{3}=\omega_{24}^{3}=\omega_{34}^{2}=\omega_{43}^{2}=0,
\end{equation} 

$(\textbf{b})$ If $b_1=0$, we have that
\begin{equation}\label{e:d44}b_i=\alpha\lambda_i+\phi,\quad e_1(\alpha)=\alpha\phi+\lambda_{1}(\epsilon_1+\alpha^2),\quad
e_1(\phi)=\phi^2+\alpha\lambda_1\phi,\end{equation} for some
smooth functions $\alpha$ and $\phi$, ($i=2,3,4$).
\end{lem}

\begin{proof}  \textbf{(a)} Evaluating $g(R(e_{1},e_{3})e_{2},e_{4})$ and
$g(R(e_{1},e_{2})e_{3},e_{4})$ using (\ref{e:a7}), (\ref{e:b1}),
Lemma \ref{L1} and Lemma \ref{w}, we find
\begin{equation}\label{e:ba9}
e_{1}(\omega_{32}^{4})-\omega_{32}^{4}b_3=0, e_{1}(\omega_{23}^{4})-\omega_{23}^{4}b_2=0,
\end{equation}
respectively.

Putting $j=4, k=2, i=3$ in (\ref{e:b7}), we get
 \begin{equation}\label{e:ba11}
(\lambda_{2}-\lambda_{4})\omega_{32}^{4}=(\lambda_{3}-\lambda_{4})\omega_{23}^{4}.
\end{equation}

Differentiating (\ref{e:ba11}) with respect to $e_{1}$, and using
(\ref{e:ba9}), we find
\begin{equation}\label{e:ba12}
(e_{1}(\lambda_{2})-e_{1}(\lambda_{4})) \omega_{32}^{4}+
(\lambda_{2}-\lambda_{4}) \omega_{32}^{4}\ b_3=(e_{1}(\lambda_{3})-e_{1}(\lambda_{4})) \omega_{23}^{4}+
(\lambda_{3}-\lambda_{4}) \omega_{23}^{4}\ b_2.
\end{equation}

Using (\ref{e:b6}) in (\ref{e:ba12}), we obtain
\begin{equation}\label{e:ba13}
\omega_{32}^{4}(b_2-b_4)=\omega_{23}^{4}(b_3-b_4).
\end{equation}

Eliminating $\omega_{23}^{4}$ using (\ref{e:ba11}) and
(\ref{e:ba13}), we find
\begin{equation}\label{e:ba14}
\omega_{32}^{4}\ b_1=0.
\end{equation}

From (\ref{e:ba14}), we obtain  $\omega_{32}^{4}=0$ as $b_1\neq 0$. Using $\omega_{32}^{4}=0$ in
(\ref{e:ba11}), gives $\omega_{23}^{4}=0.$ From (\ref{e:b5}), we get
$\omega_{34}^{2}=-\omega_{32}^{4}$ and
$\omega_{23}^{4}=-\omega_{24}^{3}$, therefore, we obtain
$\omega_{34}^{2}=0$ and $\omega_{24}^{3}=0$, which by use of
(\ref{e:b7}) gives $\omega_{43}^{2}=0$ and $\omega_{42}^{3}=0$.

\smallskip
\textbf{(b)} Since  $b_1=0$, it follows that 
\begin{equation}\label{e:d47}
\frac{b_2-b_3}{\lambda_{2}-\lambda_{3}}=\frac{b_4-b_3}{\lambda_{4}-\lambda_{3}}
=\frac{b_2-b_4}{\lambda_{2}-\lambda_{4}}=\alpha,
\end{equation}
for some smooth function $\alpha$.

From (\ref{e:d47}), we get
\begin{equation}\label{e:d48}
b_i=\alpha\lambda_{i}+\phi, \ i=2,3,4,
\end{equation}
for some smooth function $\phi$.

Differentiating (\ref{e:d48}) with respect to $e_1$ and using
(\ref{e:b6}), (\ref{e:b14}) and (\ref{e:d48}), we find
\begin{equation}\label{e:d49}
e_1(\alpha)=\alpha\phi+\lambda_{1}(\epsilon_1+\alpha^2),\quad
e_1(\phi)=\phi^2+\alpha\lambda_1\phi,
\end{equation}
whereby completing  the proof of the lemma.\end{proof}

\section{\textbf{Proof of  Theorem $1.1$}}
Evaluating
$$g(R(e_{2},e_{3})e_{2},e_{3}), \quad g(R(e_{2},e_{4})e_{2},e_{4}),
\mbox{and}\quad g(R(e_{3},e_{4})e_{3},e_{4}),$$ using (\ref{e:a7}), Lemma \ref{L1} and Lemma \ref{w}, we find that
\begin{equation}\label{la1}
-\epsilon_1 b_2 b_3
+(\omega_{32}^{4}\omega_{23}^{4}-\omega_{34}^{2}\omega_{43}^{2}-\omega_{42}^{3}\omega_{24}^{3})\epsilon_2
\epsilon_3\epsilon_4= \lambda_{2} \lambda_{3},
\end{equation}
\begin{equation}\label{la2}
-\epsilon_1b_2 b_4
+(\omega_{42}^{3}\omega_{24}^{3}-\omega_{34}^{2}\omega_{43}^{2}-\omega_{32}^{4}\omega_{23}^{4})\epsilon_2
\epsilon_3\epsilon_4= \lambda_{2} \lambda_{4},
\end{equation}
\begin{equation}\label{la3}
-\epsilon_1 b_3 b_4
+(\omega_{34}^{2}\omega_{43}^{2}-\omega_{24}^{3}\omega_{42}^{3}-\omega_{23}^{4}\omega_{32}^{4})\epsilon_2
\epsilon_3\epsilon_4= \lambda_{3} \lambda_{4}.
\end{equation}
As in  Lemma \ref{ww} we set $b_{1}=(\lambda_{3}-\lambda_{4})b_2+(\lambda_{4}-\lambda_{2})b_3+(\lambda_{2}-\lambda_{3})b_4$, and consider two cases.

\smallskip
\noindent
\textbf{Case I:} $b_1\ne 0$, so (\ref{e:d43}) holds. Then, from (\ref{la1}), (\ref{la2}), (\ref{la3}) we find
\begin{equation}\label{e:ba17}
\epsilon_1 b_2 b_3=-
\lambda_{2} \lambda_{3}, \quad \epsilon_1 b_2 b_4=- \lambda_{2} \lambda_{4}, \quad
\epsilon_1 b_3 b_4=-
\lambda_{3} \lambda_{4},
\end{equation}
respectively, and (\ref{e:b17}) can be written as

\begin{equation}\label{a1}
e_1e_1(\lambda_1)=\left(b_2+b_3+b_4\right) e_1(\lambda_1)+\lambda _1 \epsilon _1 \left(\varepsilon  \left(\lambda _1^2+\lambda _2^2+\lambda _3^2+\lambda _4^2\right)-\lambda \right).
\end{equation}

Differentiating (\ref{e:b9}) with respect to $e_1$, we find
\begin{equation}\label{a3}
b_2 \left(\lambda _2-\lambda _1\right)+b_3 \left(\lambda _3-\lambda _1\right)+b_4 \left(\lambda _4-\lambda _1\right)+3 e_1(\lambda_1)=0.
\end{equation}

Differentiating (\ref{a3}) with respect to $e_1$ and using (\ref{e:b6}),  (\ref{e:b14}), we find
\begin{eqnarray}\label{a4}
2 b_4^2 \lambda _4+3 e_1e_1(\lambda_1)+\lambda _1 \left(\lambda _2^2+\lambda _3^2+\lambda _4^2\right) \epsilon _1=2 b_2^2 \left(\lambda _1-\lambda _2\right)+2 b_4^2 \lambda _1\\+2 b_3^2 \left(\lambda _1-\lambda _3\right)+(b_2 +b_3 +b_4 )e_1(\lambda_1)+\lambda _1^2 \lambda _2 \epsilon _1+\lambda _1^2 \lambda _3 \epsilon _1+\lambda _1^2 \lambda _4 \epsilon _1.\nonumber
\end{eqnarray}

Eliminating $e_1e_1(\lambda_1)$ using (\ref{a1}) from  (\ref{a4}), we get
\begin{eqnarray}\label{a5}
-2 b_2^2 \left(\lambda _1-\lambda _2\right)-2 b_3^2 \left(\lambda _1-\lambda _3\right)+2 b_4^2 \lambda _4+2 \left(b_2+b_3+b_4\right) e_1(\lambda_1)\nonumber\\-\lambda _1 (2 b_4^2+(3 \lambda +\lambda _1 (\lambda _2  +\lambda _3+\lambda _4)) \epsilon _1)+3 \varepsilon  \lambda _1^3 \epsilon _1+3 \varepsilon  \lambda _2^2 \lambda _1 \epsilon _1+3 \varepsilon  \lambda _3^2 \lambda _1 \epsilon _1\\+3 \varepsilon  \lambda _4^2 \lambda _1 \epsilon _1+\lambda _2^2 \lambda _1 \epsilon _1+\lambda _3^2 \lambda _1 \epsilon _1+\lambda _4^2 \lambda _1 \epsilon _1=0.\nonumber
\end{eqnarray}

Eliminating $e_1(\lambda_1)$ using (\ref{a3}) from  (\ref{a5}), we obtain
\begin{eqnarray}\label{a6}
b_2 \left(b_3 \left(4 \lambda _1-2 \left(\lambda _2+\lambda _3\right)\right)+2 b_4 \left(2 \lambda _1-\lambda _2-\lambda _4\right)\right)+2 b_3 b_4 \left(2 \lambda _1-\lambda _3-\lambda _4\right)+4 b_4^2 \lambda _4\nonumber\\+9 \varepsilon  \lambda _1^3 \epsilon _1+9 \varepsilon  \lambda _2^2 \lambda _1 \epsilon _1+9 \varepsilon  \lambda _3^2 \lambda _1 \epsilon _1+9 \varepsilon  \lambda _4^2 \lambda _1 \epsilon _1+3 \lambda _2^2 \lambda _1 \epsilon _1+3 \lambda _3^2 \lambda _1 \epsilon _1+3 \lambda _4^2 \lambda _1 \epsilon _1\\=4 b_2^2 \left(\lambda _1-\lambda _2\right)+4 b_3^2 \left(\lambda _1-\lambda _3\right)+\lambda _1 \left(4 b_4^2+3 \left(3 \lambda +\lambda _1 \left(\lambda _2+\lambda _3+\lambda _4\right)\right) \epsilon _1\right).\nonumber
\end{eqnarray}

Eliminating $b_2, b_3, b_4$ and $\lambda_1$ using (\ref{e:b9}) and (\ref{e:ba17}) from  (\ref{a6}), we obtain
\begin{equation}\label{a7}
\lambda _2 \lambda _3 \lambda _4 \epsilon _1 F_4(\lambda_2, \lambda_3, \lambda_4)=0.
\end{equation}
where
\begin{eqnarray}
F_4(\lambda_2, \lambda_3, \lambda_4)&=& \varepsilon  \lambda _2^3+3 \varepsilon  \lambda _3 \lambda _2^2+3 \varepsilon  \lambda _4 \lambda _2^2+12 \varepsilon  \lambda _3^2 \lambda _2+12 \varepsilon  \lambda _4^2 \lambda _2+6 \varepsilon  \lambda _3 \lambda _4 \lambda _2+10 \varepsilon  \lambda _3^3\nonumber\\&+&10 \varepsilon  \lambda _4^3+12 \varepsilon  \lambda _3 \lambda _4^2+12 \varepsilon  \lambda _3^2 \lambda _4+4 \lambda _2^3+6 \lambda _3 \lambda _2^2+6 \lambda _4 \lambda _2^2+6 \lambda _3^2 \lambda _2+6 \lambda _4^2 \lambda _2+6 \lambda _3 \lambda _4 \lambda _2\nonumber\\&+&4 \lambda _3^3+4 \lambda _4^3+6 \lambda _3 \lambda _4^2+6 \lambda _3^2 \lambda _4-\lambda  \left(9 \lambda _2+9 \lambda _3+9 \lambda _4\right)-16 \lambda _2^3 \epsilon _1+6 \lambda _3 \lambda _2^2 \epsilon _1+6 \lambda _4 \lambda _2^2 \epsilon _1\nonumber\\&+&6 \lambda _3^2 \lambda _2 \epsilon _1+6 \lambda _4^2 \lambda _2 \epsilon _1+12 \lambda _3 \lambda _4 \lambda _2 \epsilon _1-16 \lambda _3^3 \epsilon _1-16 \lambda _4^3 \epsilon _1+6 \lambda _3 \lambda _4^2 \epsilon _1+6 \lambda _3^2 \lambda _4 \epsilon _1.\nonumber
\end{eqnarray}
Assume for a moment that $\lambda _2 \lambda _3 \lambda _4\neq 0$. Then from (\ref{a7}), we find that  
\begin{equation}\label{a8}
 F_4(\lambda_2, \lambda_3, \lambda_4)=0.
\end{equation}

Differentiating (\ref{a8}) with respect to $e_1$ and eliminating $b_2, b_3,b_3$ and $\lambda_1$ using (\ref{e:b9}) and (\ref{e:ba17}) from it, we get
\begin{equation}\label{a9}
 F_5(\lambda_2, \lambda_3, \lambda_4)=0,
 \end{equation}
where
\begin{eqnarray}
F_5(\lambda_2, \lambda_3, \lambda_4)&=&2 \varepsilon  \lambda _2^4+5 \varepsilon  \lambda _3 \lambda _2^3+5 \varepsilon  \lambda _4 \lambda _2^3+15 \varepsilon  \lambda _3^2 \lambda _2^2+15 \varepsilon  \lambda _4^2 \lambda _2^2+9 \varepsilon  \lambda _3 \lambda _4 \lambda _2^2+23 \varepsilon  \lambda _3^3 \lambda _2+23 \varepsilon  \lambda _4^3 \lambda _2\nonumber\\&+&18 \varepsilon  \lambda _3 \lambda _4^2 \lambda _2+18 \varepsilon  \lambda _3^2 \lambda _4 \lambda _2+20 \varepsilon  \lambda _3^4+20 \varepsilon  \lambda _4^4+23 \varepsilon  \lambda _3 \lambda _4^3+24 \varepsilon  \lambda _3^2 \lambda _4^2+23 \varepsilon  \lambda _3^3 \lambda _4+8 \lambda _2^4\nonumber\\&+&11 \lambda _3 \lambda _2^3+11 \lambda _4 \lambda _2^3+12 \lambda _3^2 \lambda _2^2+12 \lambda _4^2 \lambda _2^2+12 \lambda _3 \lambda _4 \lambda _2^2+11 \lambda _3^3 \lambda _2+11 \lambda _4^3 \lambda _2+12 \lambda _3 \lambda _4^2 \lambda _2\nonumber\\&+&12 \lambda _3^2 \lambda _4 \lambda _2+8 \lambda _3^4+8 \lambda _4^4+11 \lambda _3 \lambda _4^3+12 \lambda _3^2 \lambda _4^2+11 \lambda _3^3 \lambda _4-\lambda  (6 \lambda _2^2+3 \lambda _3 \lambda _2+3 \lambda _4 \lambda _2\nonumber\\&+&6 \lambda _3^2+6 \lambda _4^2+3 \lambda _3 \lambda _4)-32 \lambda _2^4 \epsilon _1+\lambda _3 \lambda _2^3 \epsilon _1+\lambda _4 \lambda _2^3 \epsilon _1+12 \lambda _3^2 \lambda _2^2 \epsilon _1+12 \lambda _4^2 \lambda _2^2 \epsilon _1\nonumber\\&+&18 \lambda _3 \lambda _4 \lambda _2^2 \epsilon _1+\lambda _3^3 \lambda _2 \epsilon _1+\lambda _4^3 \lambda _2 \epsilon _1+18 \lambda _3 \lambda _4^2 \lambda _2 \epsilon _1+18 \lambda _3^2 \lambda _4 \lambda _2 \epsilon _1-32 \lambda _3^4 \epsilon _1\nonumber\\&-&32 \lambda _4^4 \epsilon _1+\lambda _3 \lambda _4^3 \epsilon _1+12 \lambda _3^2 \lambda _4^2 \epsilon _1+\lambda _3^3 \lambda _4 \epsilon _1.\nonumber
\end{eqnarray}

Eliminating $\lambda_2$  using (\ref{a8})  from  (\ref{a9}), we obtain
\begin{eqnarray}\label{a10}
F_6(\lambda_3,\lambda_4)=0,
\end{eqnarray}
where
\begin{eqnarray}
F_6(\lambda_3,\lambda_4)&=&\lambda _3^{12} \left(518807386 \varepsilon -4 (120770281 \varepsilon +131102956) \epsilon _1+486317196\right)\nonumber\\&-&6 \lambda _3^{11} \lambda _4 \left(-98611981 \varepsilon +2 (55935542 \varepsilon +47015777) \epsilon _1-104941596\right)+\dots .\nonumber
\end{eqnarray}

Differentiating (\ref{a10}) with respect to $e_1$ and eliminating $b_3, b_4$ and $\lambda_1$ using (\ref{la2}) and  (\ref{e:ba17}) and  (\ref{a7})  from it, we get
\begin{eqnarray}\label{a11}
F_7(\lambda_3,\lambda_4)=0.
\end{eqnarray}

Eliminating $\lambda_3$  using (\ref{a10})  from  (\ref{a11}), we obtain a polynomial equation in terms of $\lambda_4$ with real coefficients, therefore $\lambda_4$ must be real constant. Using this fact in 
(\ref{a10}), we get that $\lambda_3$ is a real constant. Then, using $\lambda_3, \lambda_4$ as constants in (\ref{a9}) we obtain $\lambda_2$ a real constant. Therefore, from (\ref{e:b9}) we find that $\lambda_1$ must be constant, a contradiction to (\ref{e:b3}). Hence
$\lambda_2\lambda_3\lambda_4=0$, and
this means that at least one  $\lambda_2, \lambda_3, \lambda_4$ must be zero. If we choose 
$\lambda_2=0$ arbitrary. Then from (\ref{e:b6}), (\ref{e:b9}), and (\ref{e:ba17}), we obtain
\begin{equation}\label{a12}
b_2=0,\ 3 \lambda _1+\lambda _3+\lambda _4=0,   b_3 b_4=-\epsilon_1 \lambda_3 \lambda_4.
\end{equation}

Using (\ref{a12}) in (\ref{a6}) and eliminating $\lambda_1$ and $b_3$, we find 
\begin{eqnarray}\label{a13}
4 \left(b_4^4 \left(\lambda _3+4 \lambda _4\right)+\lambda _3^2 \lambda _4^2 \left(4 \lambda _3+\lambda _4\right)\right)=b_4^2 \left(\lambda _3+\lambda _4\right) \epsilon _1 (2 (5 \varepsilon +2) \lambda _3^2\\+2 (\varepsilon -4) \lambda _4 \lambda _3+2 (5 \varepsilon +2) \lambda _4^2-9 \lambda ).\nonumber
\end{eqnarray}

Differentiating (\ref{a13}) with respect to $e_1$, and eliminating $\lambda_1$ and $b_3$ using (\ref{a12}), we find
\begin{eqnarray}\label{a14}
b_4^2 \lambda _4 (4 (35 \varepsilon +22) \lambda _3^4+10 (17 \varepsilon +12) \lambda _4 \lambda _3^3+\lambda _3^2 \left(8 (15 \varepsilon +1) \lambda _4^2-54 \lambda \right)+\lambda _3 ((56 \varepsilon \\-4) \lambda _4^3-45 \lambda  \lambda _4)+4 (5 \varepsilon +2) \lambda _4^4-18 \lambda  \lambda _4^2)+64 b_4^6 \left(\lambda _3+4 \lambda _4\right)+b_4^4 \epsilon _1 (-4 (18 \varepsilon +5) \lambda _3^3\nonumber\\-16 (9 \varepsilon -1) \lambda _4 \lambda _3^2+\lambda _3 \left(63 \lambda -2 (99 \varepsilon +20) \lambda _4^2\right)-4 (45 \varepsilon +34) \lambda _4^3+90 \lambda  \lambda _4)\nonumber\\=8 \lambda _3^2 \lambda _4^3 \left(24 \lambda _3^2+10 \lambda _4 \lambda _3+\lambda _4^2\right) \epsilon _1.\nonumber
\end{eqnarray}

Eliminating $b_4$ using (\ref{a13}) from  (\ref{a14}), we get
\begin{equation}\label{a15}
F_8(\lambda_3,\lambda_4)=0,
\end{equation}
where
\begin{eqnarray}\label{la5}
F_8(\lambda_3,\lambda_4)&=&48 \lambda _3^3 \lambda _4^3 (4 \lambda _3^3+21 \lambda _4 \lambda _3^2+21 \lambda _4^2 \lambda _3+4 \lambda _4^3) (-4 \lambda _3^4 (162 (37 \varepsilon +5) \lambda ^2+2 (10679 \varepsilon \nonumber\\&+&634) \lambda _4^4-27 (218 \varepsilon +405) \lambda  \lambda _4^2)-4 \lambda _4 \lambda _3^3 (162 (91 \varepsilon -16) \lambda ^2+2 (2209 \varepsilon +7118) \lambda _4^4\nonumber\\&+&3 (1697-7196 \varepsilon ) \lambda  \lambda _4^2)+2 \lambda _3^2 (-2187 (13 \varepsilon -6) \lambda ^2 \lambda _4^2+8 (1157 \varepsilon +613) \lambda _4^6+54 (218 \varepsilon \nonumber\\&+&405) \lambda  \lambda _4^4+6561 \lambda ^3)+\lambda _4 \lambda _3 (-648 (91 \varepsilon -16) \lambda ^2 \lambda _4^2+16 (1036 \varepsilon +1457) \lambda _4^6-108 (202 \varepsilon \nonumber\\&-&183) \lambda  \lambda _4^4+28431 \lambda ^3)+2 (-324 (37 \varepsilon +5) \lambda ^2 \lambda _4^4+880 (10 \varepsilon +11) \lambda _4^8-12 (496 \varepsilon \nonumber\\&+&53) \lambda  \lambda _4^6+6561 \lambda ^3 \lambda _4^2)+1760 (10 \varepsilon +11) \lambda _3^8+16 (1036 \varepsilon +1457) \lambda _4 \lambda _3^7-8 \lambda _3^6 (3 (496 \varepsilon \nonumber\\ &+&53) \lambda-2 (1157 \varepsilon +613) \lambda _4^2)-4 \lambda _4 \lambda _3^5 (2 (2209 \varepsilon +7118) \lambda _4^2+27 (202 \varepsilon -183) \lambda)).\nonumber
\end{eqnarray}

Differentiating (\ref{a15}) with respect to $e_1$ and eliminating $\lambda_1, b_3, b_4$ using (\ref{a12}) and (\ref{a13})  from  it, we obtain
\begin{equation}\label{a16}
F_9(\lambda_3,\lambda_4)=0.
\end{equation}

Eliminating $\lambda_3$  using (\ref{a15})  from  (\ref{a16}), we obtain a polynomial equation in terms of $\lambda_4$ with real coefficients, therefore $\lambda_4$ must be real constant. Using this fact in 
(\ref{a15}), we get $\lambda_3$ a real constant. Then, from (\ref{a12}), we find $\lambda_1$ must be a real constant, a contradiction to (\ref{e:b3}).

\smallskip
\noindent
\textbf{Case II:} $b_1=0$, so (\ref{e:d44}) holds. Then, from (\ref{e:b7}) and (\ref{e:b5}) we obtain
\begin{equation}\label{a17}
(\lambda_{2}-\lambda_{3})\omega_{42}^{3}=(\lambda_{4}-\lambda_{3})\omega_{24}^{3}=(\lambda_{2}-\lambda_{4})\omega_{32}^{4}.
\end{equation}

From (\ref{a17}) and (\ref{e:b5}), we find
\begin{equation}\label{a18}
\omega_{34}^{2}\omega_{43}^{2}+\omega_{42}^{3}\omega_{24}^{3}+\omega_{32}^{4}\omega_{23}^{4}=0,
\end{equation}

Adding (\ref{la1}), (\ref{la2}),(\ref{la3}) and using  (\ref{a18}), we obtain
\begin{equation}\label{a19}
 b_2 b_3+ b_2 b_4+b_3 b_4=-\epsilon_1
(\lambda_{2} \lambda_{3}+ \lambda_{2} \lambda_{4}+\lambda_{3} \lambda_{4})
\end{equation}
respectively.

Using (\ref{e:d44}) and (\ref{e:b9}) in (\ref{a19}), we get
\begin{equation}\label{a20}
 -6 \alpha  \lambda _1 \phi +k \left(\alpha ^2+\epsilon _1\right)+3 \phi ^2=0,
\end{equation}
where $k=\lambda_{2} \lambda_{3}+ \lambda_{2} \lambda_{4}+\lambda_{3} \lambda_{4}$.

Now, differentiating (\ref{e:b9}) with respect to $e_1$ and using
(\ref{e:d44}), (\ref{e:b9}), we obtain
\begin{equation}\label{a21}
-3 e_1(\lambda _1)=12 \alpha  \lambda _1^2-2 \alpha  k-6 \lambda _1 \phi.
\end{equation}

Using (\ref{e:d44}), (\ref{e:b9}) in (\ref{e:b17}) , we find
\begin{equation}\label{a22}
3 \epsilon _1 e_1(\lambda _1) \left(\phi -\alpha  \lambda _1\right)+\varepsilon  \lambda _1 \left(10 \lambda _1^2-2 k\right)-\lambda  \lambda _1-\epsilon _1 e_1e_1(\lambda _1)=0.
\end{equation}

Now, differentiating the function $k$  with respect to $e_1$ and using
(\ref{e:d44}), (\ref{e:b6}), (\ref{e:b9}), we obtain
\begin{equation}\label{a23} 
e_1(k)=k \left(2 \phi -5 \alpha  \lambda _1\right)+6 \lambda _1^2 \phi -3 \alpha  P,
\end{equation}
where $P=\lambda_2\lambda_3\lambda_4$.

Now, differentiating (\ref{a21}) with respect to $e_1$ and using
(\ref{e:d44}), (\ref{a21}) and (\ref{a23}), we obtain
\begin{eqnarray}\label{a25}
-3 e_1e_1(\lambda _1)&=&2 (6 \lambda _1^3 (\epsilon _1-7 \alpha ^2)+33 \alpha  \lambda _1^2 \phi +\lambda _1 (12 \alpha ^2 k-k \epsilon _1-9 \phi ^2)\\&+&\alpha  (3 \alpha  P-5 k \phi )).\nonumber
\end{eqnarray}

Eliminating $e_1e_1(\lambda _1)$ from (\ref{a25}) using (\ref{a22}), we
get
\begin{eqnarray}\label{a26}
\lambda _1 (6 (5 \varepsilon +2) \lambda _1^2-6 \varepsilon  k-2 k-3 \lambda)\\+2 \alpha  \epsilon _1 \left(-24 \alpha  \lambda _1^3+9 \alpha  k \lambda _1+6 \lambda _1^2 \phi -2 k \phi +3 \alpha  P\right)=0.\nonumber
\end{eqnarray}

Now, differentiating (\ref{a20}) with respect to $e_1$ and using
(\ref{e:d44}), (\ref{a21}) and (\ref{a23}), we obtain
\begin{equation}\label{a27}
6 \left(3 \alpha ^2 \lambda _1^2 \phi +\phi ^3\right)+\epsilon _1 \left(-3 \alpha  k \lambda _1+2 k \phi -3 \alpha  P\right)=3 \lambda _1 \left(6 \alpha  \phi ^2+\alpha ^3 k\right)+3 \alpha ^3 P.
\end{equation}

Eliminating $P$, from (\ref{a26}) using (\ref{a27}), we get
\begin{multline}\label{a28}
\epsilon _1 (6 \lambda _1^3 \left(8 \alpha ^4-5 \varepsilon -2\right)-48 \alpha ^3 \lambda _1^2 \phi +4 \alpha  \phi  \left(\alpha ^2 k-3 \phi ^2\right)+\lambda _1 (3 \left(12 \alpha ^2 \phi ^2+\lambda \right)\\+k \left(-12 \alpha ^4+6 \varepsilon +2\right)))+\alpha  \lambda _1 \left(-6 \alpha  (5 \varepsilon -6) \lambda _1^2+3 \alpha  \lambda +2 \alpha  (3 \varepsilon -5) k-12 \lambda _1 \phi \right)=0.
\end{multline}

Eliminating $k$, from (\ref{a28}) using (\ref{a20}), we get
\begin{eqnarray}\label{a29}
2 \lambda _1^3 \left(\alpha ^4 (5 \varepsilon -14)+5 \varepsilon +2\right)+\phi ^3 \left(8 \alpha ^3 \epsilon _1+4 \alpha \right)+\phi (40 \alpha ^5 \lambda _1^2 \epsilon _1+4 \alpha ^3 (10\\-3 \varepsilon ) \lambda _1^2-12 \alpha  \varepsilon  \lambda _1^2 \epsilon _1)+\phi ^2 \left(2 \lambda _1 \epsilon _1 \left(-16 \alpha ^4+3 \varepsilon +1\right)+2 \alpha ^2 (3 \varepsilon -11) \lambda _1\right)\nonumber\\-4 \alpha ^2 \lambda _1^3 \epsilon _1 \left(4 \alpha ^4-5 \varepsilon +2\right)-\lambda  \lambda _1 \left(\alpha ^4+2 \alpha ^2 \epsilon _1+1\right)=0.\nonumber
\end{eqnarray}

Also, from (\ref{a20}) and (\ref{a21}), we have
\begin{equation}\label{a30} e_1(\lambda_1)=-\frac{2 \left(\phi -2 \alpha  \lambda _1\right) \left(\alpha ^2 \left(-\lambda _1\right)+\alpha  \phi -\lambda _1 \epsilon _1\right)}{\alpha ^2+\epsilon _1}.
\end{equation}

Now, differentiating (\ref{a29}) with respect to $e_1$ and using
(\ref{e:d44}) and (\ref{a30}), we
get 
\begin{eqnarray}\label{a31}
8 \alpha  \lambda _1^4 \epsilon _1 \left(6 \alpha ^8+\alpha ^4 (6-15 \varepsilon )-5 \varepsilon -4\right)+\phi ^4 \left(56 \alpha ^5 \epsilon _1-6 \alpha ^3 (\varepsilon -9)-6 \alpha  (\varepsilon -1) \epsilon _1\right)\\-8 \alpha ^3 \lambda _1^4 \left(\alpha ^4 (5 \varepsilon -14)+15 \varepsilon +6\right)+\phi ^2 (-2 \alpha  \lambda _1^2 \left(\alpha ^4 (63 \varepsilon -254)+33 \varepsilon +18\right)+\left(\alpha ^4+1\right) \alpha  \lambda\nonumber \\+2 \alpha ^3 \lambda  \epsilon _1+8 \alpha ^3 \lambda _1^2 \epsilon _1 \left(47 \alpha ^4-24 \varepsilon +12\right))+\phi ^3 (2 \alpha ^4 (27 \varepsilon -137) \lambda _1-2 \alpha ^2 \lambda _1 \epsilon _1 \left(124 \alpha ^4-33 \varepsilon +18\right)\nonumber\\+6 (2 \varepsilon +1) \lambda _1)+\phi  (-5 \alpha ^6 \lambda  \lambda _1-11 \alpha ^4 \lambda  \lambda _1 \epsilon _1-7 \alpha ^2 \lambda  \lambda _1-4 \lambda _1^3 \epsilon _1 \left(58 \alpha ^8+\alpha ^4 (35-73 \varepsilon )-6 \varepsilon -3\right)\nonumber\\+2 \alpha ^2 \lambda _1^3 \left(\alpha ^4 (67 \varepsilon -212)+91 \varepsilon +32\right)-\lambda  \lambda _1 \epsilon _1)=0.\nonumber
\end{eqnarray}

Now, differentiating (\ref{a31}) with respect to $e_1$ and using
(\ref{e:d44}) and (\ref{a30}), we
get 
\begin{multline}\label{a32}
-8 \lambda _1^5 \epsilon _1 \left(42 \alpha ^{12}-5 \alpha ^8 (35 \varepsilon -8)-150 \alpha ^4 (\varepsilon +1)+5 \varepsilon +4\right)+8 \alpha ^2 \lambda _1^5 (3 \alpha ^8 (15 \varepsilon -38)\\+10 \alpha ^4 (25 \varepsilon +14)+25 \varepsilon +38)+\phi ^4 (-4720 \alpha ^8 \lambda _1 \epsilon _1-2 \alpha ^4 (1198-945 \varepsilon ) \lambda _1 \epsilon _1+4 \alpha ^2 \lambda _1 (\alpha ^4 (291 \varepsilon\\ -1790)+195 \varepsilon +36)+18 (3 \varepsilon +2) \lambda _1 \epsilon _1)+\phi ^5 (10 \alpha ^5 (143-15 \varepsilon )+4 \alpha ^3 \epsilon _1 \left(250 \alpha ^4-51 \varepsilon +120\right)\\+18 \alpha  (1-3 \varepsilon ))+\phi  (-56 \alpha ^7 \lambda  \lambda _1^2 \epsilon _1-48 \alpha ^3 \lambda  \lambda _1^2 \epsilon _1-2 \alpha  \lambda _1^4 (\alpha ^8 (955 \varepsilon -2624)\\+3 \alpha ^4 (1085 \varepsilon +562)+2 (65 \varepsilon +73))-\alpha  (15 \alpha ^8+78 \alpha ^4+11) \lambda  \lambda _1^2+2 \alpha ^3 \lambda _1^4 \epsilon _1(1140 \alpha ^8+\alpha ^4 (886-2995 \varepsilon )\\-1345 \varepsilon -1234)+\phi ^2 (-144 \alpha ^6 \lambda  \lambda _1 \epsilon _1-32 \alpha ^2 \lambda  \lambda _1 \epsilon _1+2 \lambda _1^3 (2 \alpha ^8 (949 \varepsilon -3119)+\alpha ^4 (3847 \varepsilon +1454)\\+51 \varepsilon +24)-2 (29 \alpha ^8+58 \alpha ^4+1) \lambda  \lambda _1+2 \alpha ^2 \lambda _1^3 \epsilon _1 (-3136 \alpha ^8+5 \alpha ^4 (949 \varepsilon -474)+1051 \varepsilon +746))\\+\phi ^3 (8048 \alpha ^9 \lambda _1^2 \epsilon _1+17 \alpha ^7 \lambda +2 \alpha ^5 (2309-2979 \varepsilon ) \lambda _1^2 \epsilon _1+\left(39 \alpha ^4+5\right) \alpha  \lambda  \epsilon _1+27 \alpha ^3 \lambda -2 \alpha ^3 \lambda _1^2 (\alpha ^4 (1518 \varepsilon\\ -6767)+1704 \varepsilon +635)-6 \alpha  (81 \varepsilon +67) \lambda _1^2 \epsilon _1).
\end{multline}

Eliminating $\phi$ from (\ref{a31}) and  (\ref{a32})  using (\ref{a29}), we find
\begin{equation}\label{a33}
F_{10}(\alpha,\lambda_1)=0,
\end{equation}
 and 
\begin{equation}\label{a34}
F_{11}(\alpha,\lambda_1)=0,
\end{equation}
respectively, where
\begin{eqnarray}\label{e:ba42}
F_{10}(\alpha,\lambda_1)&=&\alpha ^{30} (-1024 (45 \varepsilon -16) \lambda ^3 \lambda _1^4 \epsilon _1-512 (356 \varepsilon -493) \lambda ^2 \lambda _1^6 \epsilon _1\nonumber\\&-&2048 (163 \varepsilon +8) \lambda _1^{10} \epsilon _1-1024 (287 \varepsilon -389) \lambda  \lambda _1^8 \epsilon _1+2560 \lambda ^4 \lambda _1^2 \epsilon _1)+\dots,\nonumber
\end{eqnarray}
\begin{eqnarray}\label{e:ba42}
F_{11}(\alpha,\lambda_1)&=&\alpha ^{40} (-1792 (674565 \varepsilon -3009044) \lambda ^4 \lambda _1^4 \epsilon _1-512 (194248606 \varepsilon -78896071) \lambda ^3 \lambda _1^6 \epsilon _1\nonumber\\&-&1024 (365259971 \varepsilon -525070142) \lambda ^2 \lambda _1^8 \epsilon _1-32768 (21285955 \varepsilon +1679714) \lambda _1^{12} \epsilon _1\nonumber\\&-&8192 (73514984 \varepsilon -94715669) \lambda  \lambda _1^{10} \epsilon _1+43265920 \lambda ^5 \lambda _1^2 \epsilon _1)+\dots .\nonumber
\end{eqnarray}

Eliminating $\lambda_1$ from (\ref{a33}) and (\ref{a34}), we find
\begin{eqnarray}\label{a35}
97922991388784963151200256 \alpha ^{16}\lambda ^{60}\\(2\alpha ^2(11988081807762758701231508180334280704000 \alpha ^{368}\nonumber \\(19740310929523 \varepsilon +22414659708205)+\dots))^2=0,\nonumber
\end{eqnarray}
 which is polynomial equation in terms of $\alpha$ with real coefficients. Hence $\alpha$ must be constant. Using this fact in (\ref{a33}), gives  $\lambda_1$ a constant, a contradiction to (\ref{e:b3}).
 Whereby the proof of the Theorem 1.1 is complete.

\section{\textbf{Biharmonic hypersurfaces in $E^5_s$ with diagonal shape
operator}}

In this section, we study biharmonic non-degenerate hypersurfaces
Riemannian or semi-Riemannian $M_{r}^{4}$  of diagonal shape
operator. A biharmonic submanifold in a semi-Euclidean space is
called proper
 biharmonic if it is not minimal. From biharmonic equation $\triangle \vec{H}=0,$ the necessary and sufficient
 conditions for $M_{r}^{4}$ to be biharmonic in $E_{s}^{5}$ are
\begin{equation}\label{e:d1}
\triangle H + \varepsilon H \trace \mathcal{A}^{2} = 0,
\end{equation}
 and (\ref{e:a11}).\\

\noindent
{\it Proof of Theorem 1.3}.
We consider the following cases:

\medskip
\noindent
(a) \emph{All  principal curvatures are distinct.}\\
 In this case proceeding in a similar way as
of Section 3 and Section 4 by taking $\lambda=0$, we get that $H$ is
constant. Then, from (\ref{e:d1}), we get that $H=0$.

\medskip
\noindent
(b) \emph{Three distinct principal curvatures.} \\
The case of three distinct principal curvatures for
biharmonic hypersurfaces in semi-Euclidean space $E_{s}^{5}$ has
already been treated in \cite{r19} and have also
concluded that $H$ must be zero.

\medskip
\noindent
(c) \emph{Two distinct principal curvatures.}\\
Let  $M_{r}^{4}$ be a biharmonic hypersurface of diagonal shape
operator with two distinct principal curvatures $\lambda_1$ and
$\lambda_2$ of multiplicities $n_1$ and $n_2=4-n_1$, respectively.
We assume that the mean curvature is not constant and $\grad
H\neq0$. Assuming non constant mean curvature implies the existence
of an open connected subset $U$ of $M^{4}_{r}$ with $\grad_{p}H \neq
0$ for all $p\in U$. From (\ref{e:a11}), it is easy to see that
$\grad H$ is an eigenvector of the shape operator $\mathcal{A}$ with
the corresponding principal curvature $-2\varepsilon H$. Without
lose of generality, we choose $e_{1}$ in the direction of grad$H$
and, therefore, equation (\ref{e:b3}) also holds in this case.

From (2.4), we find
\begin{equation}\label{e:d2}
n_1\lambda_1+(4-n_1)\lambda_2=4\varepsilon H.
\end{equation}

For $n_1>1$, from (\ref{e:b6}), we get that $e_1(H)=0$, which is a
contradiction of (\ref{e:b3}). Therefore, the shape operator will
take the following form with respect to a suitable frame  $\{e_{1},
e_{2}, e_{3}, e_{4}\}$
\begin{equation}\label{e:d3} \mathcal{A}e_1=-2\varepsilon He_1, \quad\mbox{and}\quad \mathcal{A}e_i=\lambda_2 e_i,  \quad \mbox{for}\quad i=2, 3, 4.
\end{equation}

From (\ref{e:d2}) and (\ref{e:d3}), we get
\begin{equation}\label{e:d4}
\lambda_2=2\varepsilon H.
\end{equation}

Also, from (\ref{e:b6}) and (\ref{e:d3}), we obtain
\begin{equation}\label{e:d5}
\epsilon_{2}\omega_{22}^{1}=\epsilon_{3}\omega_{33}^{1}=\epsilon_{4}\omega_{44}^{1}=\frac{e_{1}(H)}{2H}.
\end{equation}

Also, using (\ref{e:a7}), evaluating $R(e_{1}, e_{2}, e_{1},
e_{2})$, gives
\begin{equation}\label{e:d6}
 e_{1}(\omega_{22}^{1})\epsilon_{2}=(\omega_{22}^{1})^{2}-\epsilon_{1}4H^{2}.
\end{equation}

Using (\ref{e:d5}) and (\ref{e:d6}), we find
\begin{equation}\label{e:d7}
\epsilon_{1}e_{1}e_{1}(H)=\frac{3e_{1}^{2}(H)}{2H}-8H^{3}.
\end{equation}

On the other hand, from (\ref{e:d1}), (\ref{e:a12}), (\ref{e:d3}),
and (\ref{e:d5}), we have
\begin{equation}\label{e:d8}
\epsilon_{1}e_{1}e_{1}(H)=\frac{3e_{1}^{2}(H)}{2H}+16H^{3}.
\end{equation}

From (\ref{e:d7}) and (\ref{e:d8}), we get that $H$ must be zero.\\

Combining (a), (b) and (c), the proof of the Theorem 1.3 is
complete.


\bibliography{xbib}


\medskip
Authors' addresses:\\
\textbf{Ram Shankar Gupta }\\
 University School of Basic and Applied
Sciences,
 Guru Gobind Singh Indraprastha University,
  Sector--16C,
Dwarka, New Delhi--110078,
India.\\
\smallskip
\noindent
\textbf{Email:} ramshankar.gupta@gmail.com\\
\textbf{Andreas Arvanitoyeorgos }\\
 University of Patras, Department of Mathematics, GR-26500 Patras, Greece and\\
 Hellenic Open University, Aristotelous 18, GR-26335 Patras, Greece\\
\textbf{Email:} arvanito@math.upatras.gr\\
\end{document}